\documentclass[11pt,a4paper]{amsart}

\usepackage[T1]{fontenc}
\usepackage{graphicx}
\usepackage{comment}
\usepackage{amssymb}
\usepackage{color}
\usepackage{setspace}
\usepackage{mathtools}
\usepackage{tikz}
\usepackage{tikz-cd}
\usepackage{dsfont}
\usepackage{commath}
\usepackage{enumerate}
\usepackage{mathrsfs}
\usepackage{fancyhdr}
\usepackage{amsthm}
\usepackage{scalerel} 
\usepackage{fancyhdr}
\usepackage{bbm}
\usepackage{hyperref}
\usepackage[normalem]{ulem}
\usepackage{tikz}
\usepackage{subcaption}
\usepackage{ esint }

\usetikzlibrary{decorations.fractals}

\newtheorem{theorem}{Theorem}[section]

\newtheorem{lemma}[theorem]{Lemma}
\newtheorem{claim}[theorem]{Claim}

\newtheorem{proposition}[theorem]{Proposition}

\theoremstyle{definition}
\newtheorem{definition}[theorem]{Definition}

\theoremstyle{remark}
\newtheorem{remark}[theorem]{Remark}

\numberwithin{equation}{section}

\newcommand{\f}{\varphi}

\newcommand{\cF}{\mathcal{F}}
\newcommand{\cB}{\mathcal{B}}

\newcommand{\Z}{\mathbb{Z}}

\newcommand{\R}{\mathbb{R}}

\newcommand{\bS}{\mathbb{S}}
\newcommand{\bM}{\mathbb{M}}

\newcommand{\supp}{{\rm{spt}\,}}
\newcommand{\spt}{{\rm{spt}}}
\newcommand{\eps}{\varepsilon}

\newcommand{\diam}{\operatorname{diam}}

\newcommand{\dist}{\operatorname{dist}}

\newcommand{\rad}{\operatorname{rad}}

\newcommand{\obeta}{\mathring{\beta}}

\def\XXint#1#2#3{{\setbox0=\hbox{$#1{#2#3}{\int}$ }
\vcenter{\hbox{$#2#3$ }}\kern-.6\wd0}}

\newcommand{\epshplus}{\varepsilon_{n}^+}

\newcommand{\zji}{z_{i}}

\newcommand{\DD}{{\mathcal D}}

\newcommand{\HH}{{\mathcal H}}

\newcommand{\Ch}{{\mathcal Ch}}

\newcommand{\rf}[1]{{(\ref{#1})}}

\newcommand{\ve}{{\varepsilon}}
\newcommand{\vv}{{\vspace{2mm}}}

\newcommand{\rest}{{\lfloor}}

\newcommand{\pom}{{\partial\Omega}}

\newcommand{\sss}{{\mathsf {Stop}}}
\newcommand{\ttt}{{\mathsf {Top}}}

\newcommand{\tree}{{\rm Tree}}
\newcommand{\nex}{{\rm Next}}

\def\XXint#1#2#3{{\setbox0=\hbox{$#1{#2#3}{\int}$ }
\vcenter{\hbox{$#2#3$ }}\kern-.58\wd0}}

\usepackage[normalem]{ulem}

\definecolor{color1}{RGB}{50,180,50}
\definecolor{color2}{RGB}{50,50,180}%dark blue
\definecolor{color3}{RGB}{140, 200, 160}

\definecolor{backshade}{RGB}{150, 210, 225}
\definecolor{frontshade}{RGB}{150, 220, 240}
    \usetikzlibrary{positioning}

\newcommand\pgfmathsinandcos[3]{%
  \pgfmathsetmacro#1{sin(#3)}%
  \pgfmathsetmacro#2{cos(#3)}%
}
\newcommand\LongitudePlane[3][current plane]{%
  \pgfmathsinandcos\sinEl\cosEl{#2} % elevation
  \pgfmathsinandcos\sint\cost{#3} % azimuth
  \tikzset{#1/.style={cm={\cost,\sint*\sinEl,0,\cosEl,(0,0)}}}
}
\newcommand\LatitudePlane[3][current plane]{%
  \pgfmathsinandcos\sinEl\cosEl{#2} % elevation
  \pgfmathsinandcos\sint\cost{#3} % latitude
  \pgfmathsetmacro\yshift{\cosEl*\sint}
  \tikzset{#1/.style={cm={\cost,0,0,\cost*\sinEl,(0,\yshift)}}} %
}

\newcommand\DrawLatitudeCircle[2][1]{
  \LatitudePlane{\angEl}{#2}
  \tikzset{current plane/.prefix style={scale=#1}}
  \pgfmathsetmacro\sinVis{sin(#2)/cos(#2)*sin(\angEl)/cos(\angEl)}
  \pgfmathsetmacro\angVis{asin(min(1,max(\sinVis,-1)))}
  \draw[current plane,thin,black] (\angVis:1) arc (\angVis:-\angVis-180:1);
  \draw[current plane,thin,dashed] (180-\angVis:1) arc (180-\angVis:\angVis:1);
}

\tikzset{%
  >=latex,
  inner sep=0pt,%
  outer sep=2pt,%
  mark coordinate/.style={inner sep=0pt,outer sep=0pt,minimum size=3pt,
    fill=black,circle}%
}

\definecolor{purple}{RGB}{160, 32, 240}

\newcommand\DrawLongitudeCirclearc[2][1]{
  \LongitudePlane{\angEl}{#2}
  \tikzset{current plane/.prefix style={scale=#1}}
  \pgfmathsetmacro\angVis{atan(sin(#2)*cos(\angEl)/sin(\angEl))} %
  \draw[current plane,thick] (90:1) arc (90:180:1);
}
\newcommand\DrawLongitudeCirclearccolor[2][1]{
  \LongitudePlane{\angEl}{#2}
  \tikzset{current plane/.prefix style={scale=#1}}
  \pgfmathsetmacro\angVis{atan(sin(#2)*cos(\angEl)/sin(\angEl))} %
  \draw[current plane,very thick,purple] (90:1) arc (90:180:1);
}

\newcommand\DrawPartialLatitudeCircle[4][1]{
  \LatitudePlane{\angEl}{#2}
  \tikzset{current plane/.prefix style={scale=#1}}
  \pgfmathsetmacro\sinVis{sin(#2)/cos(#2)*sin(\angEl)/cos(\angEl)}
  \pgfmathsetmacro\angVis{asin(min(1,max(\sinVis,-1)))}
  \draw[current plane,thick,purple] (#3:1) arc (#3:#4:1);
}
\newcommand\DrawPartialLatitudeCirclethick[4][1]{
  \LatitudePlane{\angEl}{#2}
  \tikzset{current plane/.prefix style={scale=#1}}
  \pgfmathsetmacro\sinVis{sin(#2)/cos(#2)*sin(\angEl)/cos(\angEl)}
  \pgfmathsetmacro\angVis{asin(min(1,max(\sinVis,-1)))}
  \draw[current plane,thick] (#3:1) arc (#3:#4:1);
}
\newcommand\DrawLongitudeCirclearctoz[2][1]{
  \LongitudePlane{\angEl}{#2}
  \tikzset{current plane/.prefix style={scale=#1}}
  \pgfmathsetmacro\angVis{atan(sin(#2)*cos(\angEl)/sin(\angEl))} %
  \draw[current plane,very thick,green] (90:1) arc (90:131:1);
}

\tikzset{%
  >=latex,
  inner sep=0pt,%
  outer sep=2pt,%
  mark purple coordinate/.style={inner sep=0pt,outer sep=0pt,minimum size=3pt,
    fill=purple,circle}%
}

\title[Quantitative Carleson's conjecture]{Quantitative Carleson's conjecture for Ahlfors regular domains}
\author{Emily Casey}
\address{Emily Casey \newline
Department of Mathematics, University of Washington \newline
C138 Padelford Hall Box 354350, Seattle, WA 98195, USA}
\email{\href{mailto:ecasey@umn.edu}{ecasey@umn.edu}}

\author{Xavier Tolsa}
\address{Xavier Tolsa \newline
ICREA, Passeig Llu\'{i}s Companys 23 08010 Barcelona,  Catalonia; Department de Matem\`{a}tiques,
Universitat Aut\`{o}noma de Barcelona, 08193 Bellaterra, Catalonia; and Centre de Recerca Matem\`{a}tica, 08193 Bellaterra, Catalonia.}
\email{\href{mailto:xavier.tolsa@uab.cat}{xavier.tolsa@uab.cat}}

\author{Michele Villa}
\address{Michele Villa \newline
IKERBASQUE, Basque Foundation for Science 
Plaza Euskadi 5, 48009 Bilbao, Spain; Euskal Herriko Unibertsitatea/Universidad del Pa\'{i}s Vasco Enparantza Torres Quevedo Ingeniariaren, 1C, Basurtu-Zorrotza, 48013 Bilbo, Bizkaia, Spain
}
\email{\href{mailto:michele.villa@ehu.eus}{michele.villa@ehu.eus}}


\thanks{E.C. would like to thank the second named author, Xavier Tolsa, for his mentorship and generosity while she was visiting Universitat Aut\`{o}noma de Barcelona (UAB). E.C. would also like to thank the third named author, Michele Villa, for his mentorship. E.C. was partially supported by NSF grant DMS-1954545 and the European Research Council (ERC) under the European Union's Horizon 2020 research and innovation programme (grant agreement 101018680) during her stay at UAB}
    \thanks{X.T. was supported by the European Research Council (ERC) under the European Union's Horizon 2020 research and innovation programme (grant agreement 101018680). 
 Also partially supported by MICINN (Spain) under the grant PID2020-114167GB-I00, the María de Maeztu Program for units of excellence (Spain) (CEX2020-001084-M), and 2021-SGR-00071 (Catalonia). }
\thanks{M.V. was supported by the European Union MSCA IF QuReViMal, no. 101108515, and by an IKERBASQUE starting grant.}


\begin{document}
\maketitle

\begin{abstract}
   In this article, we prove a quantitative version of Carleson's $\ve^2$ conjecture in higher dimension: we characterise those Ahlfors-David regular domains in $\R^{n+1}$ for which the Carleson's coefficients satisfy the so-called strong geometric lemma.
   \end{abstract}
\tableofcontents

\section{Introduction}
Our aim in this article is to prove a quantitative version of the Carleson's $\ve^2$ conjecture in arbitrary dimensions, where David and Semmes' strong geometric lemma for $\beta$-numbers \cite{DS91} will serve as a model result.

Consider a Jordan domain $\Omega_1$ in the plane and let $x \in \partial \Omega_1$ and $r>0$. Denoting by $I_1(x,r)$ the longest open arc fully contained in $\Omega_1 \cap \partial B(x,r)$, and by $I_2(x,r)$ the corresponding arc in $\Omega_2=\R^2\setminus \overline \Omega_1$, we set 
\begin{equation*}
		\varepsilon(x,t)=\frac{1}{t}\max\{|\pi t- \mathcal{H}^1(I_1(x,t))|, |\pi t- \mathcal{H}^1(I_2(x,t))| \}.
\end{equation*}
In 1989, Chris Bishop, Lennart Carleson, John Garnett and Peter Jones \cite{BCGJ} proved that at $\HH^1$-almost all (double-sided) tangent points of the common boundary $\partial \Omega_i$ we have
\begin{equation}\label{vefinite}
	\int_0^1 \ve(x,r)^2 \, \frac{dr}{r} <+ \infty.
\end{equation}
The geometric intuition behind this is clear: $\partial \Omega_i$ looks flatter and flatter as we zoom in around a tangent point. Then $\ve(x,r)$ should decay to $0$, as the arc $I_i(x,r)$ becomes closer and closer to a semicircumference.
As reported in Bishop's thesis \cite{Bi87}, Carleson asked whether the converse is true. That question came to be known as the Carleson's $\ve^2$ conjecture. It proved rather influential, motivating for example the corresponding result for the $\beta$ coefficients\footnote{See Section \ref{s:prelim}. These coefficients are another way to measure local flatness of sets (or domains boundary).} by Bishop and Jones \cite{BJ94}. It was finally proved in \cite{JTV21}.

We introduced Carleson's $\ve^2$ conjecture. What about our quantitative `model result', David and Semmes' strong geometric lemma? Before any further explanation, a couple of definitions are in order. First: a set $E \subset \R^{n+1}$ is said to be Ahlfors-David $n$-regular, $n$-ADR for short, if for each point $x \in E$, and $0<r< \diam(E)$, $\HH^n (B(x,r) \cap E) \approx r^n$. This definition quantifies having positive and finite $n$-Hausdorff measure. Next, there is an integral and uniform version of \eqref{vefinite} for the $\beta$-coefficients, which reads
\begin{equation}\label{sgl}
	\int_{B \cap E} \int_0^{r(B)} \beta_{E,2}(x,r)^2 \, \frac{dr}{r} \,d\HH^n(x) \lesssim r(B)^n
\end{equation}
for any ball $B$ centered on $E$ (see \eqref{eqdefbeta} for the precise definition of $\beta_{E,2}(x,r)$). The geometric conclusion to be drawn from $n$-ADR and \eqref{sgl} is that $E$ is uniformly $n$-rectifiable (UR) (this is, in fact, a characterisation, again see \cite{DS91}). Recall that a set $E \subset \R^{n+1}$ is $n$-rectifiable if $\HH^n(E)<\infty$ and there exists a countable family of Lipschitz functions $f_i: \R^n  \to \R^{n+1}$ such that 
\begin{equation*}
	\HH^n \bigg( E \setminus \bigcup_{i} f_i(\R^n) \bigg) = 0.
\end{equation*}
It is true, in particular, that for any $x \in E$, $r>0$, there exists a Lipschitz function $f_i$ so that $\HH^n(E \cap B(x,r) \cap f_i(\R^n)) >0$.
Uniform rectifiability is a quantitative strengthening of this: given two constants $L \geq 1$, $\theta >0$, it asks that for each point $x \in E$ and $0<r<\diam(E)$, there exists a Lipschitz function $f: \R^n \supset B(0,r)\to \R^{n+1}$ with Lipschitz constant $\leq L$ so that 
\begin{equation*}
	\HH^n(E \cap B(x,r) \cap f(B(0,r))) \geq \theta r^n.
\end{equation*}

We have described our model result, which should now be reformulated in terms of the $\ve$-coefficients. For planar Jordan domains, however, $1$-ADR of the boundary immediately implies $1$-UR, without further hypotheses. The question of a strong geometric lemma for the $\ve$ coefficients, then, is not very interesting. 

In \cite{FTV23} and \cite{FTV24} Fleschler, together with the second and third named authors, introduced a higher dimensional analogue of $\ve$, which from now on we refer to as $a$. Its definition, which is coming soon, is in terms of first Dirichlet eigenvalues of domains but it has a very clear geometric significance. Indeed, in the plane, $\ve \approx a$. This computation might be found in \cite{FTV23}, Page 9, but see also \cite{AKN22}. For general domains in $\R^{n+1}$, it is no longer true that Ahlfors $n$-regularity of $\partial \Omega$ implies $n$-UR without further hypotheses. This makes our problem - whether a strong geometric lemma for the $a$ coefficients might hold - rather more interesting. Indeed, its solution is our first result:

\begin{theorem}\label{t:main1}
	Let $\Omega \subset \R^{n+1}$ be an open set, and suppose that $\partial \Omega$ is $n$-ADR. Then $\Omega$ is a two-sided corkscrew open set, and thus UR, if and only if there exists a constant $C\geq 1$ such that 
	\begin{equation}
		\int_{B\cap \partial \Omega} \int_{0}^{r(B)} a(x,r) \, \frac{dr}{r} \, d \HH^n(x) \leq C r(B)^n
	\end{equation}
	for every ball $B$ centered on $\partial \Omega$. 
\end{theorem}

We now proceed to define $a$, together with another coefficient introduced in \cite{FTV23, FTV24}, there named $\ve_n$ (not to be confused with the `simple' $\ve$). We remark that the coefficient $a$ in Theorem \ref{t:main1} is associated with $\Omega_1=\Omega$ and $\Omega_2=\R^{n+1}\setminus
\overline{\Omega_1}$.

\subsection{Definition of $a$: spherical domains and their characteristic constants}
Given a bounded open set $V$ in a Riemannian manifold $\mathbb M^n$ (such as $\R^n$ or $\mathbb{S}^n$), we say that $u\in W^{1,2}_0(V)$ is a Dirichlet eigenfunction of
$V$ for the Laplace-Beltrami operator $\Delta_{\bM^n}$ if $u\not\equiv 0$ and
$$-\Delta_{\bM^n} u = \lambda\,u,$$
for some $\lambda\in \R\setminus \{0\}$. The number $\lambda$ is the eigenvalue associated to $u$. It is well known that all the eigenvalues of the Laplace-Beltrami operator are positive and the smallest one, i.e., the first eigenvalue $\lambda_V$, satisfies
\begin{equation}\label{eqlambda1*}
	\lambda_V = \inf_{u\in  W^{1,2}_0(V)} \frac{\int_V|\nabla u|^2\,dx}{\int_V|u|^2\,dx}.
\end{equation}

Further, the infimum is attained by an eigenfunction $u$ which does not change sign, and so which can be assumed to be non-negative. 
Also, from \rf{eqlambda1*} we infer that, if $U,V\subset \mathbb M^n$ are open, then
\begin{equation}\label{eqVV'}
	U\subset V \quad \Rightarrow \quad \lambda_U \geq \lambda_{V}.
\end{equation}
In the case $\mathbb M^n=\bS^n$, to be sure the one of interest here, the characteristic constant of $V$ is defined as the positive number $\alpha_V$ such that $\lambda_V = \alpha_V(n-1+\alpha_V)$. Indeed, we now specialise our discussion to $\bS^n$.

Given two disjoint open sets $\Omega_1,\Omega_2\subset \R^{n+1}$ and $x\in\R^{n+1}$, $r>0$, put $S(x,r):= \partial B(x,r)$ and consider
the sets
$V_i(x,r) =\{r^{-1}(x-y) \, : \, y \in S(x,r)\cap \Omega^i\}$. We then define
\begin{equation}\label{eqalfai}
	\alpha_i(x,r) := \alpha_{V_i(x,r)}.
\end{equation}
By the Friedland-Hayman inequality \cite{FH}, it turns out that 

$$\alpha_1(x,r) + \alpha_2(x,r) - 2 \geq0.$$
The aforementioned computation shows that, in the plane
\begin{equation*}
	\ve(x,r)^2 \approx \min \left\{ 1, \alpha_1(x,r) + \alpha_2(x,r) -2 \right\}.
\end{equation*}
The presence of the minimum here is due to the fact that as $V_i(x,r)$ grow small, $\alpha_i(x,r)$ tends to infinity. Thus set
\begin{equation*}
	a(x,r) := \min \left\{ 1, \alpha_1(x,r) + \alpha_2(x,r) -2 \right\}.
\end{equation*}

\subsection{Definition of $\ve_n$: a more explicitly geometric coefficient}
The attentive reader might remember what was said above: that $a$ has a `very clear geometric significance'. She might now be confused as to what we mean by `clear geometric significance'. To clarify, let us introduce the further coefficient $\ve_n$, through which we'll amend our expository shortcomings.

Given two arbitrary disjoint Borel sets $\Omega_1$, $\Omega_2$ $\subset \mathbb{R}^{n+1}$, and $x\in \mathbb{R}^{n+1}$, $r>0$, define 
\begin{equation*}
	\varepsilon_n(x,r):=\frac{1}{r^n} \inf_{H^+}\mathcal{H}^n\left(((\partial B(x,r)\cap H^+)\setminus \Omega_1) \cup ((\partial B(x,r)\cap H^-)\setminus \Omega_2)\right), 
\end{equation*}
where the infimum is taken over all open affine half-spaces $H^+$ such that $x\in \partial H^+$ and $H^-=\mathbb{R}^{n+1}\setminus \overline{H^+}$. A minute's thought will clarify the geometric significance of this coefficient: if $\Omega_1$ is a half-space and $\Omega_2$ its complement, then $\varepsilon_n \equiv 0$ on the common boundary. Moreover, if we compute $\ve_n$ for a Jordan domain $\Omega$ in the plane and its complement, then one may check that $\ve_n \lesssim \ve$.

In any case, what binds $a$ and $\ve_n$ together is the following theorem, which substitutes the rather more direct computation in the plane, already mentioned above. 

\begin{theorem}[\cite{FTV24}]\label{teofac***}
	Let $V_1,V_2\subset \bS^n$ be disjoint relatively open sets and let $\ve_n(0,1)$ be defined as above, with $\Omega_i$
	replaced by $V_i$. Let $\alpha_i=\alpha_{V_i}$ for $i=1,2$.
	Then
	$$\ve_n(0,1)^2 \lesssim a(0,1).$$
\end{theorem}

\noindent
Of course, this theorem implies that, given two disjoint open subsets $\Omega_1, \Omega_2 \subset \R^{n+1}$, $x \in \R^{n+1}$ and $r>0$, we have $\ve_n (x,r)^2 \lesssim a(x,r)$. 

Having said this, let us state a more complete version of our main result. 

\begin{theorem}\label{t:main2}
	Let $\Omega_1$ and $\Omega_2$ be two disjoint open subsets of $\R^{n+1}$. Suppose that $\mu$ is an $n$-ADR measure with $\spt(\mu) = \partial \Omega_1 \cup \partial \Omega_2$. Then the following are equivalent.
	\begin{enumerate}
		\item $\Omega_1$ and $\Omega_2$ are complementary two-sided corkscrew open sets, and in particular $\mu$ is uniformly $n$-rectifiable.
		\item There is a constant $C_1$ so that for each ball $B$ centered on $\spt(\mu)$ it holds
		\begin{equation*}%\label{e:main1}
			\int_B \int_0^{r(B)} \ve_n (x,r)^2 \, \tfrac{dr}{r} \, d\mu(x) \leq C_1 r(B)^n.
		\end{equation*}
		\item There is a constant $C_1$ so that for each ball $B$ centered on $\spt(\mu)$ we have
		\begin{equation*}%\label{main2}
			\int_B \int_{0}^{r(B)} a(x,r) \, \tfrac{dr}{r} \, d\mu(x) \leq C_1 r(B)^n.
		\end{equation*}
	\end{enumerate}
\end{theorem}

Note that, 
in view of Theorem \ref{teofac***}, the implication $(1) \implies (2)$ follows at once from $(1) \implies (3)$. However, we present below a direct proof which, we believe, is of standalone interest.

Let us highlight that in \cite{Ca24}, the first author proved that having quantitative control on the rate of decay of the Carleson $\varepsilon$-function at every point of the boundary of a Jordan domain $\Omega$ gives quantitative information about the regularity of $\partial \Omega$. A little more precisely, she showed that if $\ve(x,r) \lesssim r^\alpha$ for all $x$ in some Jordan curve $\Gamma$, then $\Gamma$ is in fact a $C^{1, c(\alpha)}$ manifold. Theorem \ref{t:main2}, then, clarifies what happens ``in between'' the hypotheses of \cite{JTV21} and \cite{Ca24}.

\subsection{An open question} 
After the results of \cite{JTV21}, \cite{FTV23} and of the current article, a main issue that remains open is that of analogues in higher-codimension. Of course, to formulate them, one should think of a plausible $\ve$ coefficient. But there is a perhaps more fundamental issue: both \cite{JTV21} and \cite{FTV23} use in a fundamental way compactness arguments which naturally lead to the study of an analytic variety: this approach seems altogether unfeasible in higher codimensions. The methods we present here, however, do not entail such arguments. A way forward in higher codimensions then, is to try to obtain a quantitative statement first, in line with what we present here.

\subsection{Structure of the article and description of the proofs}
In Section \ref{s:prelim} we set out some basic notation and definitions which will be used throughout the article. In Sections \ref{s:(2)implies(1)} and \ref{s:(1)implies(2)} we show that $(1) \iff (2)$ in Theorem \ref{t:main2}. In Section \ref{s:CAD} we prove some direct estimates for CADs, aimed at the proof of $(1) \implies (3)$ in Theorem \ref{t:main2}. This proof will be completed in Section \ref{s:corona} via a corona type construction. 
%\begin{comment}

We now aim to give the reader a brief heuristic overview of the proofs contained in each section. In Section \ref{s:(2)implies(1)} we show that $(2) \implies (1)$ in Theorem \ref{t:main2}. The key idea is that from the control on the coefficients $\ve_n$ one can deduce that in any ball centered on $\partial \Omega_1\cup \partial \Omega_2$, a large portion in Lebesgue measure is contained in $\Omega_1$ and another big portion is in $\Omega_2$. From the Ahlfors regularity of $\mu$ one can upgrade the statement about large measure to the existence corkscrew balls in $\Omega_1$ and $\Omega_2$.

In Section \ref{s:(1)implies(2)} we show that $(1) \implies (2)$ in Theorem \ref{t:main2}. The uniform rectifiability of $\mu$ implies a Carleson type estimate for the $\obeta$-numbers, so that an estimate like the one in (2) holds with $\ve_n$  replaced by the $\obeta$-numbers (see Section \ref{s:(1)implies(2)} for definition of the $\obeta$-numbers). We then prove (2) by bounding $\ve_n$ by the $\obeta$-numbers. 

In Section \ref{s:CAD} we prove that if $\Omega_1$ and $\Omega_2$ are complementary \emph{chord-arc domains}, then Theorem \ref{t:main2} $(1) \implies (3)$ holds. In this particular case, the additional connectedness property of chord-arc domains allows us to utilize estimates on harmonic measure which are not available in the case of two-sided corkscrew open sets.

What then remains to be proven in Section \ref{s:corona}, is Theorem \ref{t:main2} $(1) \implies (3)$ in the case where $\Omega_1$ and $\Omega_2$ are two-sided corkscrew open sets. To this end, we use a corona type decomposition of $\partial\Omega_1$ which allows $\Omega_1$ and $\Omega_2$ to be approximated by Lipschitz subdomains. See Section \ref{secorona}.

Lipschitz domains are in particular, chord-arc domains, and thus the estimates from Section \ref{s:CAD} apply. Of course, these approximations hold only in a particular range of scales, but, by zooming in on $\spt \mu$ and iterating this construction these estimates are ``pushed down'' to $\spt \mu$. This argument is made precise in Section \ref{s:mainestimate}.

%\end{comment}

\subsection*{Acknowledgment}
Some of the mathematics of Section \ref{s:(1)implies(2)} was worked on more than a year ago by the second and third named authors together with I. Fleschler. We also thank A. Chang for useful conversation on a first version of this preprint and S. Ford for her help creating the pictures. The authors also thank the anonymous referees for their many helpful comments and suggestions which greatly improved the presentation of this article.

\section{Preliminaries}\label{s:prelim}

\subsection{Basic notation}
In the paper, constants denoted by $C$ or $c$ depend just on the dimension unless otherwise stated. As per usual, we will write $a\lesssim b$ if there is $C>0$ such that $a\leq Cb$. We write $a\approx b$ if $a\lesssim b\lesssim a$.
 
 Open balls in $\R^{n+1}$ centered in $x$ with radius $r>0$ are denoted by $B(x,r)$, and closed balls by 
$\bar B(x,r)$. For an open or closed ball $B\subset\R^{n+1}$ with radius $r$, we write $\rad(B)=r$. We use the two notations $S(x,r)\equiv \partial B(x,r)$ for spheres
in $\R^{n+1}$ centered in $x$ with radius $r$, so that $\bS^n=S(0,1)$.
If $A \subset \R^{n+1}$ is a set and $s>0$, we denote by $A(s)$ its $s$-neighbourhood, that is: $A(s)=\{y \in \R^{n+1}\,:\, \dist(y, A) < s\}$.

\subsection{Tangent points}
The notion of tangent points of domains is usually construed when they are complementary. In our case, however, 
it is appropriate to consider a somewhat more general notion involving two disjoint domains. For a point $x\in\R^{n+1}$, a unit vector $u$,
and an aperture parameter $a\in(0,1)$ we consider the two sided cone with axis in the direction of $u$ defined by
$$X_a(x,u)=\bigl\{y\in\R^{n+1}:|(y-x)\cdot u|> a|y-x|\bigr\}.$$
Given disjoint open sets $\Omega_1,\Omega_2\subset\R^{n+1}$ and $x\in\partial\Omega_1\cap\partial\Omega_2$,
we say that $x$ is a tangent (or cone) point for the pair $\Omega_1,\Omega_2$ if $x\in\pom_1\cap\pom_2$ and there exists a unit vector $u$ such that, for all
$a\in(0,1)$, there exists some $r>0$ such that
$$(\partial \Omega_1\cup \partial \Omega_2)\cap X_a(x,u)  \cap B(x,r) =\varnothing,$$
and moreover, one component of $X_a(x,u)\cap B(x,r)$ is contained in $\Omega_1$ and the other in $\Omega_2$.
The hyperplane $L$ orthogonal to $u$ through $x$ is called a tangent hyperplane at $x$.  In case that $\Omega_2=\R^{n+1}\setminus\overline{\Omega_1}$, we say that $x$ is a tangent point for $\Omega_1$.

\begin{figure}[h]
    \centering

\begin{tikzpicture}[scale=1,every node/.style={minimum size=1cm}]
%%%%%% Code based on Tomas M. Trzeciak's `Stereographic and cylindrical map projections example`: 
%  http://www.texample.net/tikz/examples/map-projections/
% AND
% Marco Miani's work 'Spherical and cartesian grids example:
% https://texample.net/spherical-and-cartesian-grids/

%% some definitions

\def\R{4} % sphere radius

\def\angEl{10} % elevation angle
\def\angAz{-100} % azimuth angle
\def\angPhiOne{-110} % longitude of point P
\def\angPhiTwo{-45} % longitude of point Q
\def\angBeta{30} % latitude of point P and Q

%% working planes

\pgfmathsetmacro\H{\R*cos(\angEl)} % distance to north pole
\LongitudePlane[xzplane]{\angEl}{\angAz}
\LatitudePlane[equator]{\angEl}{0}
\fill[ball color=white!10] (0,0) circle (\R); % 3D lighting effect

\DrawLatitudeCircle[\R]{0} % equator

%%%%%% drawing curve %%%%%%%

 \def\eps{1.2}
	\coordinate (D) at (-4,{.2*\eps});
	\def\d{0};

	\coordinate (E) at (-2,{-.6*\eps-.5});
	\def\e{30};
	\coordinate (F) at (-2.4,{.4*\eps-.3});
	\def\f{210};
    %%%% equator points for front cusp (F) plus ...
    \coordinate (G0) at (-1.75,{-.15*\eps-.5});
	  \def\gz{91};
    \coordinate (G1) at (-1.04,{-.17*\eps-.5});
 \def\ge{310};

 %\foreach \x /\y in {F/\f,G0/\gz,G1/\ge} {\tangy{\x}{\y}\labely{\x}{\y}} %checks coordinate/angle for each point

	\coordinate (G) at (-1,{-.2*\eps-.5});
	\def\g{-50};
    
	\coordinate (H) at (0,{0*\eps-.4});
	\def\h{30};
    %%%%%% equator points for bump including (H)
    \coordinate (H1) at (-.46,{-0.25*\eps-.4});
	\def\ha{35};
    \coordinate (H2) at (.44,{-0.24*\eps-.4});
	\def\hz{122};

 % \foreach \x /\y in {H/\h,H1/\ha,H2/\hz} {\tangy{\x}{\y}\labely{\x}{\y}} %checks coordinate/angle for each point
    
	\coordinate (I) at (.8,{-.1*\eps-.8});
	\def\i{20};
    
	\coordinate (J) at (2.6,{.3*\eps-.4});
	\def\j{10};
	\coordinate (K) at (3.2,{0*\eps-.4});
	\def\k{-90};
%%%%%%%% equator points for bump along with (J) and (K)
\coordinate (K1) at (1.35,{-0.2*\eps-.4});
	\def\ka{35};

% \foreach \x /\y in {J/\j,K/\k,K1/\ka} {\tangy{\x}{\y}\labely{\x}{\y}} %checks coordinate/angle for each point
    
    \coordinate (L) at (3,{-.6*\eps-.4});
	\def\l{-90};
	\coordinate (M) at (3.88,{-.8*\eps});
	\def\m{15};
    
 \def\addsm{.9}
	\coordinate (T) at (-2,{-.6*\eps+\addsm});
	\def\t{40};
    
	\coordinate (U) at (-2.4,{.4*\eps+\addsm});
	\def\u{-50};
%%%%%%%% equator points for back cusp with (U)
	\coordinate (U1) at (-1.93,{-.25*\eps+\addsm});
	\def\ua{105};
    \coordinate (U2) at (-1.85,{-.24*\eps+\addsm});
	\def\uz{121};

% \foreach \x /\y in {U/\u,U1/\ua,U2/\uz} {\tangy{\x}{\y}\labely{\x}{\y}} %checks coordinate/angle for each point

	\coordinate (V) at (0,{0*\eps-.5});
	\def\v{30};
    
	\coordinate (W) at (0,{0*\eps+\addsm});
	\def\w{30};
%%%%%%%%%% equator points for back pocket 1 with (W)
	\coordinate (W1) at (-0.16,{-0.18*\eps+\addsm});
	\def\wa{80};
    \coordinate (W2) at (0.6,{-0.18*\eps+\addsm});
	\def\wz{143};
% \foreach \x /\y in {W1/\wa,W/\w,W2/\wz} {\tangy{\x}{\y}\labely{\x}{\y}} %checks coordinate/angle for each point
    
	\coordinate (X) at (.8,{-.1*\eps+\addsm-.1});
	\def\x{20};
	\coordinate (Y) at (2.6,{.3*\eps+\addsm});
	\def\y{10};
%%%%%%%% equator points for back pocket 2 with (Y) and (Z)
\coordinate (X1) at (.85,{-.1*\eps+\addsm-.1});
	\def\xa{22};
\coordinate (Y1) at (3.23,{-0.4*\eps+\addsm});
	\def\ya{95};

% \foreach \x /\y in {X1/\xa,Y/\yy,Y1/\ya} {\tangy{\x}{\y}\labely{\x}{\y}} %checks coordinate/angle for each point
    
	\coordinate (Z) at (3.2,{0*\eps+\addsm});
	\def\z{-90};

    \coordinate (A) at (-2,2);
	\def\a{45};
	\coordinate (B) at (-1.8,2.2);
	\def\b{30};
	\coordinate (C) at (-1.6,2.4);
	\def\c{10};
	\coordinate (N1) at (-1.4,2.3);
	\def\n{0};
	\coordinate (O1) at (-1,2.2);
	\def\o{-100};
	\coordinate (P) at (-1.5,2.1);
	\def\p{200};
    \coordinate (R) at (-1.9,1.8);
	\def\r{190};
    \coordinate (S1) at (-1.2,2.5);
	\def\s{-20};

\filldraw[nearly transparent, fill=frontshade] %%%transparent shading for front
    (M) arc (-13.7:175:4) to
	(D) to [out=\d,in={\e-180}] 
	(E) to [out=\e,in={\f-180}] 
	(F) to [out=\f-180,in={\g-180}] 
	(G) to [out=\g,in={\h-180}] 
	(H) to [out=\h,in={\i-180}] 
	(I) to [out=\i,in={\j-180}] 
	(J) to [out=\j,in={\k-180}] 
	(K) to [out=\k,in={\l-180}] 
	(L) to [out=\l,in={\m-180}] 
    (M) -- cycle;

\filldraw[semitransparent, fill=red] %draws red front cusp 
   (-1.75,-.68) to [out={\gz}, in={\f-180}]
   (-2.4,.18) to  [out={\f-180}, in={\ge-180}]
   (-1.04,-0.69) to [out={178}, in={3}] (-1.75,-.64);	

\filldraw[semitransparent, fill=red] %draws red front pocket 2 
   (-.46,{-0.25*\eps-.4}) to [out={\ha}, in={\h-180}]
   (0,{0*\eps-.4}) to  [out={\h}, in={\hz}]
   (.44,{-0.24*\eps-.4}) to [out={178}, in={3}] (-.46,{-0.25*\eps-.4});	

\filldraw[semitransparent, fill=red] %draws red front pocket 3 
   (1.35,{-0.2*\eps-.4}) to [out={\ka}, in={\j-180}]
   (2.6,{.3*\eps-.4}) to  [out={\j}, in={\k-180}]
   (3.2,{0*\eps-.42}) to [out={192}, in={0.5}] (1.35,{-0.2*\eps-.4});	

\filldraw[nearly transparent, fill=red] %draws back cusp
   (-1.94,{-.25*\eps+\addsm}) to [out={\ua}, in={\u-180}]
   (-2.4,{.4*\eps+\addsm}) to  [out={\u}, in={\uz-180}]
   (-1.85,{-.24*\eps+\addsm}) to [out={185}, in={0.5}] (-1.94,{-.25*\eps+\addsm});

\filldraw[nearly transparent, fill=red] %draws back pocket 1
   (-0.16,{-0.18*\eps+\addsm}) to [out={\wa}, in={\w-180}]
   (0,{0*\eps+\addsm}) to  [out={\w}, in={\wz-180}]
   (0.6,{-0.18*\eps+\addsm}) to [out={180}, in={0}] (-0.16,{-0.18*\eps+\addsm});

\filldraw[nearly transparent, fill=red] %draws back pocket 1
   (.85,{-.1*\eps+\addsm-.1}) to [out={\xa}, in={\y-180}]
   (2.6,{.3*\eps+\addsm}) to  [out={\y}, in={\z-180}]
   (3.2,{0*\eps+\addsm}) to [out={\z}, in={\ya-180}]
   (3.23,{-0.4*\eps+\addsm}) to [out={171}, in={0}] (.85,{-.1*\eps+\addsm-.1});

    \filldraw[nearly transparent, fill=backshade, draw=backshade] %%%transparent shading for back
    (M) arc (-13.7:175:4) to
	(D) to [out=\d,in={\t-180}] 
	(T) to [out=\t,in={\u-180}] 
	(U) to [out=\u,in={\v-180}] 
	(V) to [out=\v,in={\w-180}] 
	(W) to [out=\w,in={\x-180}] 
	(X) to [out=\x,in={\y-180}] 
	(Y) to [out=\y,in={\z-180}] 
	(Z) to [out=\z,in={\m-180}] 
    (M) -- cycle;

 \draw[very thick,blue] %draws front edge
	(D) to [out=\d,in={\e-180}] 
	(E) to [out=\e,in={\f-180}] 
	(F) to [out=\f-180,in={\g-180}] 
	(G) to [out=\g,in={\h-180}] 
	(H) to [out=\h,in={\i-180}] 
	(I) to [out=\i,in={\j-180}] 
	(J) to [out=\j,in={\k-180}] 
	(K) to [out=\k,in={\l-180}] 
	(L) to [out=\l,in={\m-180}] 
    (M);
  \draw[very thick,dashed,blue] %draws dashed back edge
	(D) to [out=\d,in={\t-180}] 
	(T) to [out=\t,in={\u-180}] 
	(U) to [out=\u,in={\v-180}] 
	(V) to [out=\v,in={\w-180}] 
	(W) to [out=\w,in={\x-180}] 
	(X) to [out=\x,in={\y-180}] 
	(Y) to [out=\y,in={\z-180}] 
	(Z) to [out=\z,in={\m-180}] 
    (M);
\filldraw[ semitransparent, fill=red] %shades hole
	(A) to [out=\a,in={\b-180}] 
	(B) to [out=\b,in={\c-180}] 
	(C) to [out=\c,in={\n-180}] 
	(N1) to [out=\n,in={\s-180}] 
    (S1) to [out=\s,in={\o-180}] 
	(O1) to [out=\o,in={\p-180}] 
	(P) to [out=\p,in={\r-180}] 
	(R) to [out=\r,in={\a-180}] (A);

\draw[very thick,blue] %draws hole outline
     (A) to [out=\a,in={\b-180}] 
	(B) to [out=\b,in={\c-180}] 
	(C) to [out=\c,in={\n-180}] 
	(N1) to [out=\n,in={\s-180}] 
    (S1) to [out=\s,in={\o-180}] 
	(O1) to [out=\o,in={\p-180}] 
	(P) to [out=\p,in={\r-180}] 
	(R) to [out=\r,in={\a-180}] (A);

\coordinate[mark coordinate] (O) at (0,0);
\coordinate (N) at (0,{\H+.05});
\coordinate (S) at (0,{-\H-.05});

%\node[above=8pt] at (N) {$\mathbf{p^+}$};
%\node[below=8pt] at (S) {$\mathbf{p^-}$};

\node at (0.2,.2) {$x$};

\filldraw[ semitransparent, fill=red] %(D) arc (179:182.7:4) 
(-4,0) arc (180:219: 4 cm and .7 cm) to[out=125, in=\d] (D) --cycle;%draws ellipse with upper half dashed; inputs = {color}{circle center}{radius of circle in cm}{ellipse height in cm}{line thickness}
\end{tikzpicture}
    \caption[The region $(\partial B(x,r)\cap H^+)\setminus \Omega_1$ is denoted in red.]{The region $(\partial B(x,r)\cap H^+)\setminus \Omega_1$ is denoted in red.
    \footnotemark}
    \label{fig:function on a shell}
\end{figure}
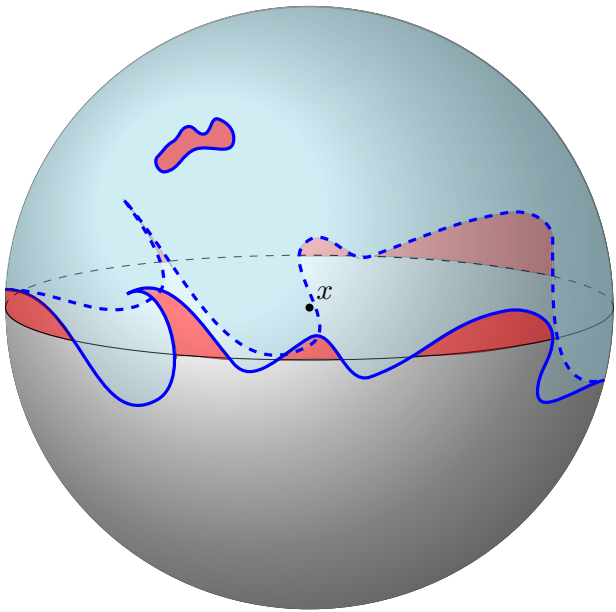

\subsection{Square functions}
In this subsection we re-define precisely $\ve_n$ and $a$ (for future reference), the geometric coefficients which are the subjects of our study.

    Given two arbitrary disjoint Borel sets $\Omega_1$, $\Omega_2$ $\subset \mathbb{R}^{n+1}$, and $x\in \mathbb{R}^{n+1}$, $r>0$, define 
    \begin{equation*}
        \varepsilon_n(x,r)=\frac{1}{r^n} \inf_{H^+}\mathcal{H}^n\left(((\partial B(x,r)\cap H^+)\setminus \Omega_1) \cup ((\partial B(x,r)\cap H^-)\setminus \Omega_2)\right), 
    \end{equation*}
    where the infimum is taken over all open affine half-spaces $H^+$ such that $x\in \partial H^+$ and $H^-=\mathbb{R}^{n+1}\setminus \overline{H^+}$. See Figure \ref{fig:function on a shell}.\\
\footnotetext{All of the pictures in this article were created using code based on \cite{trzeciak2008stereographic,miani2009spherical}.}
Let us now look at $a$. To do so, we need first a key concept, that of characteristic constant.
Given an open set $U\subset \mathbb{S}^{n}$, the \emph{characteristic constant} $\alpha=\alpha_U$ is the number which satisfies 
    \begin{equation*}
        \lambda_1(U)=\alpha(n-1+\alpha), 
    \end{equation*}
    where $\lambda_1(U)$ is the first Dirichlet-Laplacian eigenvalue of $U$. It follows from \cite{FH} that for $U,V\subset \mathbb{S}^n$ open and disjoint, 
    \begin{equation*}
        \alpha_{U}+\alpha_V-2\geq 0, 
    \end{equation*}
    where $\alpha_{U}+\alpha_V-2=0$ if and only if $U,V$ are complementary half-spheres on $\mathbb{S}^n$.\\

   Now, suppose $\Omega_1,\Omega_2\subset \mathbb{R}^{n+1}$ are open and disjoint. For $x\in \mathbb{R}^{n+1}$ and $r>0$, let $V_i=\Omega_i\cap \partial B(x,r)$, for $i=1,2$. Denote by $V_i(x,r)$ the rescaled domains on $\mathbb{S}^n$,
   \begin{equation*}
       V_i(x,r)=\left\{\frac{y-x}{r}: y\in \partial B(x,r)\cap \Omega_i\right\},
   \end{equation*}
   and note that $V_1(x,r)\cap V_2(x,r)=\varnothing$. Let $\alpha_i(x,r):=\alpha_{V_i(x,r)}$. Define 
   \begin{equation}\label{e:def-a-function}
       a(x,r):=\min\{1,\, \alpha_1(x,r)+\alpha_2(x,r)-2\}.
   \end{equation}

We remark that if $V_i=\varnothing$ for either $i=1$ or $i=2$, then $a(x,r)=1$. Indeed, this is the point of using a minimum. It might happen otherwise that $\alpha(x,r) \to \infty$.

\subsection{Ahlfors-David regularity, UR, Carleson measures}

    A Borel measure on $\mathbb{R}^{n+1}$ is said to be \emph{Ahlfors-David $n$-regular} if there exists some constant $C>0$ such that 
    \begin{equation}\label{e:def-ADR}
        C^{-1}r^n\leq \mu(B(x,r))\leq Cr^n\qquad \text{for all}\, x\in \spt\, \mu, \, r>0.
    \end{equation}

\vspace{0.5cm}
    A measure $\mu$ is said to be \emph{uniformly $n$-rectifiable} if it is $n$-ADR and there exist constants $\theta,M>0$ so that the following holds for each $x \in \spt(\mu)$ and $r>0$. There is a Lipschitz mapping $g$ from the $n$-dimensional ball $B_n(0,r)\subset \mathbb{R}^n$ to $\mathbb{R}^d$ such that $g$ has Lipschitz norm bounded by $M$ and 
    \begin{equation}\label{e:def-UR}
        \mu(B(x,r)\cap g(B_n(0,r)))\geq \theta r^n.\\
    \end{equation}

    A \emph{Carleson measure} on $E\times (0,\infty)$ is a measure $\mu$ for which there exists a constant $C>0$ such that for every $x\in E$ and $r>0$ we have 
    \begin{equation}\label{e:def-Carleson}
\int_0^r\int_{B(x,r)}d\mu(y,t)\leq Cr^n.
    \end{equation}
   
\vv
\subsection{Types of domain}
Let $\Omega\subset \mathbb{R}^{n+1}$ be an open set. $\Omega$ satisfies the \textit{$c$-corkscrew condition} if there exists some $c>0$ such that for all $x\in \partial \Omega$ and $r\in (0,\mathrm{diam}(\partial\Omega))$ there exists some ball $B\subset \Omega \cap \overline{B(x,r)}$ with $r(B)\geq cr$. 

Next, we say that $\Omega$ satisfies the \textit{two-sided $c$-corkscrew condition} (but we will usually avoid explicitly mentioning $c$) if both $\Omega$ and its complement satisfy the $c$-corkscrew condition. Thus
$\Omega$ is a \textit{two-sided corkscrew open set} if it is an open set that satisfies the two-sided corkscrew condition.

\begin{remark}
    In general, if two disjoint open subsets are not complementary, then we will write $\Omega_1,\Omega_2$. If, on the other hand, they are complementary, we will denote by $\Omega^+$ and $\Omega^-$, as customary.
\end{remark}

\begin{definition}[Harnack chain condition]\label{def: harnack chain}
    A set $\Omega\subset \mathbb{R}^{n+1}$ satisfies the \emph{Harnack chain condition} if there is some uniform constant $C>0$ such that for every $\rho>0$, $\Lambda\geq 1$, and for every pair of points $X,X^{\prime}\in \Omega$ with $d(X,\partial \Omega), d(X^{\prime}, \partial \Omega)>\rho$ and $|X-X^{\prime}|<\Lambda \rho$, there is a chain of open balls $B_1,\hdots, B_N\subset \Omega$, $N\leq C(\Lambda)$ with $X\in B_1$ and $X^{\prime}\in B_N$, $B_k\cap B_{k+1}\neq \varnothing$ and 
    \begin{equation*}
        C^{-1}\mathrm{diam}(B_k)\leq d(B_k,\partial \Omega)\leq C\mathrm{diam}(B_k).
    \end{equation*}
    The chain of balls is called a \emph{Harnack chain}.
\end{definition}

\begin{definition}[NTA domain]\label{def: NTA}
    A domain $\Omega$ is a \emph{non-tangentially accessible (NTA) domain} if $\Omega$ satisfies both the corkscrew and Harnack chain conditions, and if $\mathbb{R}^{n+1}\setminus \overline{\Omega}$ also satisfies the corkscrew condition.
\end{definition}
\begin{definition}[CAD]\label{def: cad}
    A domain $\Omega$ is a \emph{chord-arc domain (CAD)} if it is an NTA domain with $n$-ADR boundary.
\end{definition}

%\begin{comment}
    \begin{remark}\label{r:domains}
     We emphasize that the difference between a two-sided corkscrew open set with $n$-ADR boundary and CADs is the interior Harnack chain condition. This extra connectedness property enjoyed by CADs is crucial for the results of Section \ref{s:CAD}.
    \end{remark}
%\end{comment}

\subsection{Dyadic lattices}

Given an $n$-AD-regular measure $\mu$ in $\R^{n+1}$, we consider 
the dyadic lattice of ``cubes'' built by David and Semmes in \cite[Chapter 3 of Part I]{DS93}. The properties satisfied by $\DD_\mu$ are the following. 
Assume first, for simplicity, that $\diam(\supp\mu)=\infty$. Then for each $j\in\Z$ there exists a family $\DD_{\mu,j}$ of Borel subsets of $\supp\mu$ (the dyadic cubes of the $j$-th generation) such that:
\begin{itemize}
\item[$(a)$] each $\DD_{\mu,j}$ is a partition of $\supp\mu$, i.e.\ $\supp\mu=\bigcup_{Q\in \DD_{\mu,j}} Q$ and $Q\cap Q'=\varnothing$ whenever $Q,Q'\in\DD_{\mu,j}$ and
$Q\neq Q'$;
\item[$(b)$] if $Q\in\DD_{\mu,j}$ and $Q'\in\DD_{\mu,k}$ with $k\leq j$, then either $Q\subset Q'$ or $Q\cap Q'=\varnothing$;
\item[$(c)$] for all $j\in\Z$ and $Q\in\DD_{\mu,j}$, we have $2^{-j}\lesssim\diam(Q)\leq2^{-j}$ and $\mu(Q)\approx 2^{-jn}$;
\item[$(d)$] there exists $C>0$ such that, for all $j\in\Z$, $Q\in\DD_{\mu,j}$, and $0<\tau<1$,
\begin{equation}\label{small boundary condition}
\begin{split}
\mu\big(\{x\in Q:\, &\dist(x,\supp\mu\setminus Q)\leq\tau2^{-j}\}\big)\\&+\mu\big(\{x\in \supp\mu\setminus Q:\, \dist(x,Q)\leq\tau2^{-j}\}\big)\leq C\tau^{1/C}2^{-jn}.
\end{split}
\end{equation}
This property is usually called the {\em small boundaries condition}.
From (\ref{small boundary condition}), it follows that there is a point $x_Q\in Q$ (the center of $Q$) such that $\dist(x_Q,\supp\mu\setminus Q)\gtrsim 2^{-j}$ (see \cite[Lemma 3.5 of Part I]{DS93}).
\end{itemize}
We set $\DD_\mu:=\bigcup_{j\in\Z}\DD_{\mu,j}$. 

In case that $\diam(\supp\mu)<\infty$, the families $\DD_{\mu,j}$ are only defined for $j\geq j_0$, with
$2^{-j_0}\approx \diam(\supp\mu)$, and the same properties above hold for $\DD_\mu:=\bigcup_{j\geq j_0}\DD_{\mu,j}$.

Given a cube $Q\in\DD_{\mu,j}$, we say that its side length is $2^{-j}$, and we denote it by $\ell(Q)$. Notice that $\diam(Q)\leq\ell(Q)$. 
We also denote 
\begin{equation}\label{defbq}
B(Q):=B(x_Q,c_1\ell(Q)),\qquad B_Q = B(x_Q,\ell(Q)),
\end{equation}
where $c_1>0$ is some fixed
constant so that $B(Q)\cap\supp\mu\subset Q$, for all $Q\in\DD_\mu$. Clearly, we have $Q\subset B_Q$.
 We denote by $\Ch(Q)$ (the children of $Q$) the family of the cubes  from $\DD_{\mu,j+1}$ which are contained in $Q$.

For $\lambda>1$, we write
$$\lambda Q = \bigl\{x\in \supp\mu:\, \dist(x,Q)\leq (\lambda-1)\,\ell(Q)\bigr\}.$$

The side length of a ``true cube'' $P\subset\R^{n+1}$ is also denoted by $\ell(P)$. On the other hand, given a ball $B\subset\R^{n+1}$, its radius is denoted by $r(B)$. For $\lambda>0$, the ball $\lambda B$ is the ball concentric with $B$ with radius $\lambda\,r(B)$.

\subsection{The other geometric coefficient: $\beta$}
Given $E\subset\R^{n+1}$, a ball $B$, and a hyperplane $L$, we denote
$$b\beta_{E}(B,L) =  \sup_{y\in E\cap B} \frac{\dist(y,L)}{r(B)} + 
\sup_{y\in L\cap B}\!\! \frac{\dist(y,E)}{r(B)} .$$
We set
$$b\beta_{E}(B) = \inf_L b\beta_{E}(B,L),$$
where the infimum is taken over all hyperplanes $L\subset\R^{n+1}$. For a $B=B(x,r)$, we also write
$$b\beta_{E}(x,r,L)=b\beta_{E}(B,L),\qquad b\beta_{E}(x,r)=b\beta_{E}(B).$$

For $p\geq1$, a measure $\mu$, a ball $B$, and a hyperplane $L$, we set
$$\beta_{\mu,p}(B,L) = \left(\frac1{r(B)^n}\int_B \left(\frac{\dist(x,L)}{r(B)}\right)^p\,d\mu(x)\right)^{1/p}.$$
We define
$$\beta_{\mu,p}(B) = \inf_L \beta_{\mu,p}(B,L),$$
where the infimum is taken over all hyperplanes $L$.
For $B=B(x,r)$, we also write
\begin{equation}\label{eqdefbeta}
\beta_{\mu,p}(x,r,L) = \beta_{\mu,p}(B,L),\qquad \beta_{\mu,p}(x,r) = \beta_{\mu,p}(B).
\end{equation}
For $E=\supp\mu$, we may also write $\beta_{E,p}$ instead of $\beta_{\mu,p}$.
For a given cube $Q\in\DD_\mu$, we define:
\begin{align*}
\begin{array}{ll}
\beta_{\mu,p}(Q,L)  = \beta_{\mu,p}(B_Q,L)
,&\quad \beta_{\mu,p}(\lambda Q,L)= \beta_{\mu,p}(\lambda B_Q,L),\\
\quad \,\beta_{\mu,p}(Q) = \beta_{\mu,p}(B_Q),
& \quad \quad\,\beta_{\mu,p}(\lambda Q)= \beta_{\mu,p}(\lambda B_Q).
\end{array}
\end{align*}
Also, we define similarly
$$b\beta_\mu(Q,L),\quad b\beta_\mu(\lambda Q,L),\quad b\beta_\mu(Q),\quad b\beta_\mu(\lambda Q),$$
by identifying these coefficients with the analogous ones in terms of $B_Q$.
These coefficients are defined in the same way as $b\beta_{\supp\mu}(B,L)$ and $b\beta_{\supp\mu}(B)$,
replacing again $B$ by $Q\in\DD_\mu$ or $\lambda Q$.

The coefficients $b\beta_E$ and $\beta_{\mu,p}$ above measure the goodness of the approximation of $E$ and
$\supp\mu$, respectively, in a ball $B$ by a hyperplane. They play an important role in the theory of
uniform $n$-rectifiability. See \cite{DS91}.

\subsection{The ACF monotonicity formula}
Recall that the Alt-Caffarelli-Friedman (ACF) monotonicity formula asserts the following:
\begin{theorem} \label{teoACF-elliptic}  Let $x \in  \R^{n+1}$ and $R>0$. Let $u_1,u_2\in
		W^{1,2}(B(x,R))\cap C(B(x,R))$ be nonnegative subharmonic functions such that $u_1(x)=u_2(x)=0$ and $u_1\cdot u_2\equiv 0$. 
		Set
		\begin{equation}\label{eqACF2}
			J(x,r) = \left(\frac{1}{r^{2}} \int_{B(x,r)} \frac{|\nabla u_1(y)|^{2}}{|y-x|^{n-1}}dy\right)\cdot \left(\frac{1}{r^{2}} \int_{B(x,r)} \frac{|\nabla u_2(y)|^{2}}{|y-x|^{n-1}}dy\right)
		\end{equation}
		Then $J(x,r)$ is an absolutely continuous function of $r\in (0,R)$ and
			\begin{equation}\label{eqprec1}
			\frac{\partial_rJ(x,r)}{J(x,r)}\geq \frac2r\bigl(\alpha_1 + \alpha_2  - 2 \bigr). 
		\end{equation}
		where $\alpha_i$ is the characteristic constant of the open subset  $\Omega_i\subset\bS^n$ given by  $$ \Omega_i=\bigl\{r^{-1}(y-x): y\in\partial B(x,r),\,u_i(y)>0\bigr\}.$$
		Further, for $r\in (0,R/2)$ and $i=1,2$, we have 
		\begin{equation}\label{eqaux*}
		\frac{1}{r^{2}} \int_{B(x,r)} \frac{|\nabla u_i(y)|^{2}}{|y-x|^{n-1}}dy\lesssim \frac1{r^{n+1}}\|\nabla u_i\|_{L^2(B(x,2r))}^2.
		\end{equation}	
	\end{theorem}
\vv

\section{From square function estimates to corkscrews}\label{s:(2)implies(1)}

In this section we prove the implication $(2) \implies (1)$, which, together with Theorem \ref{teofac***} also immediately gives that $(3) \implies (1)$. More precisely, our aim here will be to demonstrate the following proposition. 

\begin{proposition}\label{p:ven-carleson-corkscrew}
    Let $\Omega_i$, $i=1,2$ be two disjoint open subsets of $\R^{n+1}$. Suppose that $\mu$ is an $n$-ADR measure with $\spt(\mu) = \partial \Omega_1 \cup \partial \Omega_2$. If there exists a constant $C_1$ so that for each ball $B$ centered on $\spt(\mu)$ it holds
    \begin{equation}\label{e:ve-n-Carleson}
        \int_B \int_0^{r(B)} \ve_n(x,r)^2 \, \frac{dr}{r} \, d\mu(x) \leq C_1 r(B)^n,
    \end{equation}
    then $\Omega_i$, $i=1,2$ are complementary two-sided corkscrew open sets, and in particular $\mu$ is uniformly $n$-rectifiable. 
\end{proposition}

\begin{proof}[Proof of Proposition \ref{p:ven-carleson-corkscrew}]
To show that $\mathbb{R}^{n+1}\setminus \overline{\Omega_1}=\Omega_2$, first we will prove that $\partial\Omega_1=\partial \Omega_2$. Indeed, if this fails, renaming $\Omega_1$ and $\Omega_2$ if necessary, then there exists a point $x\in \partial \Omega_1\setminus \partial \Omega_2$ with $d(x,\partial \Omega_2)>0$. Let $r$ be such that $0<r<\frac{d(x,\partial \Omega_2)}{2}$. Then, since $\Omega_1$ and $\Omega_2$ are disjoint, $B(x,r)\subset \mathbb{R}^{n+1}\setminus\overline{\Omega_2}$, and thus 
\begin{equation}\label{e: eps big if bdry not the same}
    \varepsilon_n(z,s)\approx 1 \qquad \text{for}\qquad 0<s\leq r, \;z\in B(x,r/2).
\end{equation}
In fact, since $\mu$ is $n$-ADR, \eqref{e: eps big if bdry not the same} holds for a positive $\mu$-measure subset of $\partial \Omega_1 \cap B(x,r/2)$, contradicting \eqref{e:ve-n-Carleson}. Thus, $\pom_1=\pom_2$.

Next we prove $\mathbb{R}^{n+1}\setminus \overline{\Omega_1}=\Omega_2$, or equivalently, that $\mathbb{R}^{n+1}=\overline{\Omega_1}\cup\overline{\Omega_2}$. Denote $F= \overline{\Omega_1}\cup\overline{\Omega_2} =\overline{\Omega_1\cup\Omega_2}$ and let $V= \mathbb{R}^{n+1}\setminus F$, so that $\partial V= \partial F\subset \pom_1=\pom_2$. From \cite[Lemma 10.1]{FTV24} and the assumption \eqref{e:ve-n-Carleson}, it follows that for $\HH^n$-a.e. $x\in\pom_1$, 
  $$\lim_{r\to0}\frac{\HH^{n+1}(B(x,r)\cap \Omega_1)}{\HH^{n+1}(B(x,r))}  = \lim_{r\to0}\frac{\HH^{n+1}(B(x,r)\cap \Omega_2)}{\HH^{n+1}(B(x,r))} = \frac12$$
  and so  $$\lim_{r\to0}\frac{\HH^{n+1}(B(x,r)\cap F)}{\HH^{n+1}(B(x,r))} =1.$$
This means that for $\HH^n$-a.e. $x\in \partial V$
$$
\lim_{r\to0}\frac{\HH^{n+1}(B(x,r)\cap V)}{\HH^{n+1}(B(x,r))} =0,
$$
and thus $\HH^n(\partial^*V)=0$, where $\partial^*V$ denotes the set of $x\in \partial V$ such that both 
$$
\lim_{r\to0}\frac{\HH^{n+1}(B(x,r)\cap V)}{\HH^{n+1}(B(x,r))} >0 \qquad \text{and}\qquad \lim_{r\to0}\frac{\HH^{n+1}(B(x,r)\setminus V)}{\HH^{n+1}(B(x,r))} >0.
$$
By the Ahlfors $n$-regularity of $\partial \Omega_1 =\partial \Omega_2$,  $\HH^n(\partial V)$ is locally finite; hence $V$ is a set of locally finite perimeter, see for instance \cite[Proposition 3.62]{ambrosio2000functions}. Then by the local isoperimetric inequality it turns out that given $\tau\in (0,1)$, for any ball  $B(x,r)$  such that $\HH^{n+1}(V\cap B(x,r))< \tau\HH^{n+1}(B(x,r))$, it holds that
$$
0=\HH^n(\partial^* V\cap B(x,r))\gtrsim_{n,\tau} \HH^{n+1}(V\cap B(x,r))^{\frac{n-1}{n}}
$$
(see \cite[Proposition 12.37]{maggi2012sets}).
Thus $\HH^{n+1}(V)=0$. Since $V$ is open, this means that $V=\varnothing$. Thus $\mathbb{R}^{n+1}=\overline{\Omega_1}\cup\overline{\Omega_2}$.
% Denote $F= \overline{\Omega_1}\cup\overline{\Omega_2} =\overline{\Omega_1\cup\Omega_2}$ and, for contradiction, suppose that $V= \mathbb{R}^{n+1}\setminus F\neq\varnothing$, so that $\partial F= \partial V\subset \pom_1=\pom_2$.
%  By the local isoperimetric inequality it turns out that $\HH^n(\partial V)>0$. Further, by the $n$-ADR of $\pom_1$, it follows that $V$ and $F$ are sets with locally finite perimeter (see \cite[p.222]{evans2015measure}, for example), and hence for
%  $\HH^n$-a.e. $x\in\partial V$, we have
%  \begin{equation}\label{eqvomega00}
%  \lim_{r\to0}\frac{\HH^{n+1}(B(x,r)\cap V)}{\HH^{n+1}(B(x,r))}  = \lim_{r\to0}\frac{\HH^{n+1}(B(x,r)\cap F)}{\HH^{n+1}(B(x,r))} = \frac12.
%  \end{equation}
% However, from \cite[Lemma 10.1]{FTV24} and the assumption \eqref{e:ve-n-Carleson}, one easily deduces that for $\HH^n$-a.e. $x\in\pom_1$, 
%  $$\lim_{r\to0}\frac{\HH^{n+1}(B(x,r)\cap \Omega_1)}{\HH^{n+1}(B(x,r))}  = \lim_{r\to0}\frac{\HH^{n+1}(B(x,r)\cap \Omega_2)}{\HH^{n+1}(B(x,r))} = \frac12$$
%  and so  $$\lim_{r\to0}\frac{\HH^{n+1}(B(x,r)\cap F)}{\HH^{n+1}(B(x,r))} =1,$$
% which contradicts \eqref{eqvomega00}, and thus $\mathbb{R}^{n+1}=\overline{\Omega_1}\cup\overline{\Omega_2}$.

It remains to prove that $\Omega_1$ and $\Omega_2$ satisfy the corkscrew condition.
Fix $m\in \mathbb{N}$ sufficiently large to be chosen later. For any $x_0\in \spt(\mu)$ and $0<R<\mathrm{diam}\,\spt(\mu)$, by \eqref{e:ve-n-Carleson} applied to $B(x_0, R)$, there exists some point $x\in B(x_0, R)\cap\supp(\mu)$ such that
\begin{equation}\label{e:123}
    \sum_{k\geq0} \,\int_{2^{-k-1}R}^{2^{-k}R} \varepsilon_n(x,r)^2\frac{dr}{r} =
    \int_0^R \varepsilon_n(x,r)^2\frac{dr}{r}\leq C_1.
\end{equation}
 Let $\delta=C_1/m$. By the preceding estimate, there exists some $K\in [m,2m]$ such that
 \begin{equation}\label{e: epsilon int with delta}
 \int_{2^{-K-1}R}^{2^{-K}R} \varepsilon_n(x,r)^2\frac{dr}{r}\leq  \delta.
\end{equation} 
Indeed, if not, then for all $K\in [m,2m]$ we would have $\int_{2^{-K-1}R}^{2^{-K}R}\frac{\varepsilon_n(x,r)^2}{r}dr>\delta$, which contradicts \eqref{e:123}.
 
    Let $B_0:=B(x,2^{-K}R)$ and $A_0:=A(x,2^{-K-1}R,2^{-K}R)$. We first claim that
    \begin{equation}\label{e: B contains a lot of mass on both sides}
        \mathcal{H}^{n+1}(A_0\cap \Omega_1)\geq \frac{1}{4}\mathcal{H}^{n+1}(A_0) \,\, \mbox{ and } \,\,  \mathcal{H}^{n+1}(A_0\cap \Omega_2)\geq \frac{1}{4}\mathcal{H}^{n+1}(A_0)
    \end{equation}
        To see this, recall that we denote by $H^+$ an infimizing half-space in $\varepsilon_n(x,t)$, and by $H^-=\R^{n+1}\setminus \overline{H^+}$ its complementary half space. Put $H^\pm_{x,t}:=H^{\pm} \cap \partial B(x,t)$.
           We apply \eqref{e: epsilon int with delta}
and, for $m$ sufficiently large depending only on $n$ and $C_1$, (and thus $\delta>0$ sufficiently small), we compute
    \begin{align*}
        \left|\mathcal{H}^{n+1}(\Omega_1\cap A_0)-\frac{1}{2}\mathcal{H}^{n+1}(A_0) \right|
       & = \left| \int_{2^{-K-1}R}^{2^{-K}R}\left(\mathcal{H}^{n}(\Omega_1\cap S(x,t))-\frac{1}{2}\mathcal{H}^{n}(S(x,t))\right)dt\right|\\
       &\leq\left| \int_{2^{-K-1}R}^{2^{-K}R}\left(\mathcal{H}^{n}(H^+_{x,t}\setminus \Omega_2)+\mathcal{H}^{n}(H^-_{x,t}\setminus\Omega_1)\right)dt\right|\\
       &\leq \left| \int_{2^{-K-1}R}^{2^{-K}R}t^n \varepsilon_n(x,t)dt\right|\\
       &\leq \left(\int_{2^{-K-1}R}^{2^{-K}R}\frac{\varepsilon_n(x,t)^2}{t}dt \right)^{1/2}\left(\int_0^{2^{-K}R} t^{2n+1}dt\right)^{1/2}\\
        &\leq \frac{1}{4}\mathcal{H}^{n+1}(A_0), 
        \end{align*}
where the last estimate follows choosing $m$ large enough depending only on $C_1$ and $n$. The second estimate in \rf{e: B contains a lot of mass on both sides} is proven analogously.

    For $\tau\in (0,1/10)$ and $s=\tau 2^{-K-1}R$, consider the family of balls 
    \begin{equation}
        \cF := \{ B(y,s):y\in 2B_0\cap \supp(\mu) \}.
    \end{equation}
    Notice that all the balls in $\cF$ are contained in $3B_0$. By Vitali's covering theorem, there is a disjoint subfamily $\cF_0\subset\cF$ such that
    $$(\spt(\mu))(s)\cap A_0\subset \bigcup_{B\in \cF} B \subset \bigcup_{B\in \cF_0} 5B.$$
    Then, using the AD-regularity of $\mu$, we deduce
    \begin{align*}
    \HH^{n+1}((\spt(\mu))(s)\cap A_0) & \leq \sum_{B\in \cF_0} \HH^{n+1}(5B) \lesssim s\sum_{B\in \cF_0} r(B)^n\lesssim s\sum_{B\in \cF_0}
    \mu(B) \\
    & \lesssim s\,\mu(3B_0) \lesssim \tau (2^{-K-1}R)^{n+1} \approx \tau\,\HH^{n+1}(A_0).
    \end{align*}
     Thus, %for $\tau>0$ sufficiently small, 
     we obtain
    \begin{equation*}
        \HH^{n+1} \left( A_0 \setminus (\spt(\mu)) (s) \right) \geq \frac{9}{10} \HH^{n+1}(A_0).
    \end{equation*}
    This, together with \eqref{e: B contains a lot of mass on both sides} implies that
    \begin{equation*}
       \left( A_0 \setminus (\spt(\mu))(s) \right) \cap \Omega_1 \neq \varnothing,
    \end{equation*}
    and the same holds for $\Omega_2$. But note that if $x \in \left( A_0 \setminus (\spt(\mu))(s) \right) \cap \Omega_1$, then $B(x, \tau 2^{-K-1}R) \subset \Omega_1$, and in particular, from the bounds on $K$, we infer that $B(x, \tau 2^{-2m-1}R) \subset \Omega_1$. The same can be said for $\Omega_2$. The two balls thus found, one in $\Omega_1$ and the other in $\Omega_2$, are the sought after corkscrew balls. We conclude that $\Omega_1$ and its complement $\Omega_2$ are both two sided corkscrew open sets. It follows from \cite{DJ90} and \cite{Semmes90} that $\spt(\mu)$ is uniformly $n$-rectifiable. 
\end{proof}

\begin{remark}
    The argument we proposed above is substantially easier than that used in \cite{FTV23} to find (quasi)corkscrew balls. This is due to two key assumptions: that $\spt(\mu) = \partial \Omega_1 \cup \partial \Omega_2$ - as opposed to containment - and the $n$-ADR of $\mu$.
\end{remark}

\section{A direct bound of $\ve_n$ in terms of $\beta$-type coefficients}\label{s:(1)implies(2)}
In this section we prove that $(1) \implies (2)$. Of course this would follow from $(1) \implies (3)$ and Theorem \ref{teofac***}. Here, however, we prove a direct upper bound for $\ve_n$ in terms of centered $\obeta$ coefficients. This gives the desired result because, if $\mu$ is assumed to be UR, these latter coefficients satisfy the strong geometric lemma. That is, $\obeta_{\mu,2}(x,r)^2\frac{drd\mu(x)}{r}$ is a Carleson measure on $\spt\mu\times (0,\diam(\spt\mu))$, or equivalently
$$\int_{B(x_0,R)} \int_0^R \obeta_{\mu,2}(x,r)^2\frac{drd\mu(x)}{r} \lesssim \mu(B(x_0,R))$$
 for all $x_0\in\supp\mu$ and $R\in (0,\diam(\spt\mu))$.
Let $\mu$ be an $n$-Ahlfors regular measure in $\R^{d}$. Recall from Section \ref{s:prelim} that for $x\in\supp\mu$, $r>0$, 
$$
\beta_{\mu,2}(x,r) = \left(\inf_{L} \frac1{r^n}\int_{B(x,r)} \left(\frac{\dist(y,L)}{r}\right)^2\,d\mu(y)\right)^{1/2},
$$
where the infimum is taken over all $n$-planes in $\R^d$. Relevant to the proof of $(1) \implies (2)$ are the centered $\obeta$ coefficients, which we now define.
 
\begin{definition}
For $x\in \spt(\mu)$ and $r>0$, define
    \begin{equation*}
        \obeta_{\mu,2}(x,r)=\left(\inf_{L \ni x}\frac{1}{r^n}\int_{B(x,r)} \left(\frac{\dist(y,L)}{r} \right)^2d\mu(y) \right)^{1/2},
    \end{equation*}
    where the infimum is taken over all $n$-planes in $\mathbb{R}^d$ containing $x$. 
\end{definition}
Now, the strong geometric lemma is usually formulated in terms of non-centered $\beta$ coefficients (see \cite{DS91}). 
However, it also holds for the $\obeta$'s. That is:

\begin{lemma}\label{lem: centered beta and UR}
    Suppose $E\subset \mathbb{R}^{n+1}$ is a set and $\mu$ is an $n$-dimensional AD-regular measure with $\spt  \mu=E$. Then $\obeta_{\mu,2}(x,r)^2\frac{drd\mu(x)}{r}$ is a Carleson measure on $E\times (0,\diam E)$ if and only if $\mu$ is uniformly rectifiable.
\end{lemma}

Although the preceding result is folklore knowledge, for the reader's convenience we will provide the detailed proof. 
Since $\beta_{\mu,2}(x,r)\leq \obeta_{\mu,2}(x,r)$, the ``only if'' direction follows immediately from \cite{DS91}. The necessary condition is an immediate corollary of the following lemma and \cite{DS91}.

\begin{lemma}
Let $\mu$ be an $n$-ADR measure in $\R^{d}$. For all $x_0\in\supp(\mu)$ and $0<r\leq R\leq\diam(\supp\mu)$,
$$\int_{B(x_0,R)}\obeta_{\mu,2}(x,r)^2 \,d\mu(x)\lesssim \int_{B(x_0,2R)}\beta_{\mu,2}(x,2r)^2 \,d\mu(x).$$
\end{lemma}

\begin{proof}
For any $z\in B(x,r)$, denote by $L_{z,2r}$ an $n$-plane that minimizes $\beta_{\mu,2}(z,2r)$. Let $L_{z,2r}^x$ be the $n$-plane parallel to $L_{z,2r}$
through $x$.
Observe that for any $y\in B(x,r)$,
$$\dist(y,L_{z,2r}^x)\leq \dist(y,L_{z,2r}) + \dist(x,L_{z,2r}).$$
Thus, taking into account that $B(x,r)\subset B(z,2r)$,
\begin{align*}
\obeta_{\mu,2}(x,r)^2 & \leq  \frac1{r^n}\int_{B(x,r)} \left(\frac{\dist(y,L_{z,2r}^x)}{r}\right)^2\,d\mu(y)\\
& \lesssim \frac1{r^n}\int_{B(x,r)} \left(\frac{\dist(y,L_{z,2r})}{r}\right)^2\,d\mu(y) + 
\left(\frac{\dist(x,L_{z,2r})}{r}\right)^2\\
& \lesssim \beta_{\mu,2}(z,2r)^2 + 
\left(\frac{\dist(x,L_{z,2r})}{r}\right)^2.
\end{align*}
Then, averaging with respect to $z\in B(x,r)$, we obtain
\begin{align*}
\obeta_{\mu,2}(x,r)^2  
& \lesssim \frac1{r^{n}}\int_{z\in B(x,r)}\beta_{\mu,2}(z,2r)^2\,d\mu(z) + \frac1{r^n}\int_{z\in B(x,r)}
\left(\frac{\dist(x,L_{z,2r})}{r}\right)^2d\mu(z).
\end{align*}
Fix $x_0$ and $R>0$ as in the statement of the lemma. By Fubini, we obtain
\begin{align*}
\int_{B(x_0,R)}\obeta_{\mu,2}(x,r)^2 \,d\mu(x) & \lesssim 
 \frac1{r^{n}}\int_{B(x_0,R)} \int_{z\in B(x,r)}\beta_{\mu,2}(z,2r)^2\,d\mu(z)\,d\mu(x) \\
&\quad  + \frac1{r^n}
 \int_{B(x_0,R)} \int_{z\in B(x,r)}
\left(\frac{\dist(x,L_{z,2r})}{r}\right)^2 \,d\mu(z)\,d\mu(x)\\
& \lesssim
 \int_{z\in B(x_0,2R)} \frac{\mu(B(z,r))}{r^n}\,\beta_{\mu,2}(z,2r)^2\,d\mu(z)\\
&\quad +  \frac1{r^{n}}
 \int_{z\in B(x_0,2R)} 
 \int_{x\in B(z,r)}
\left(\frac{\dist(x,L_{z,2r})}{r}\right)^2 \,d\mu(x)\,d\mu(z)\\
& \lesssim \int_{z\in B(x_0,2R)} \beta_{\mu,2}(z,2r)^2\,d\mu(z).
\end{align*}
\end{proof}

\vspace{1cm}

Having dealt with this preliminary fact, we turn to prove what matters in this section:
\begin{lemma}\label{lem: eps controlled by beta}
    Let $\Omega^+\subset \mathbb{R}^{n+1}$ be a two-sided corkscrew open set and let $\Omega^-:=\mathbb{R}^{n+1}\setminus \overline{\Omega^+}$. Suppose $\mu$ is an $n$-dimensional AD-regular measure with $\spt(\mu)=\partial \Omega^+$. Then 
    \begin{equation}\label{e: eps bounded by obeta}
        \int_{r/2}^r \varepsilon_n(x,t)^2\frac{dt}{t}\lesssim \obeta_{\mu,2}(x,s)^2 \qquad \text{for all}\quad s\in\left(\frac{5}{4}r,2r\right), \, x\in \spt(\mu).
    \end{equation}
\end{lemma}

Let us prove that $(1) \implies (2)$ in Theorem \ref{t:main2} by assuming Lemma \ref{lem: eps controlled by beta} holds.
\begin{proof}[Proof of Theorem \ref{t:main2}, $(1) \implies (2)$]
Fix $x_0\in \spt(\mu)$ and $R\in (0,\diam\,\spt(\mu))$. It is sufficient to show that $I, II\leq C \mu(B(x_0,R))$ for some absolute constant $C$, where, 
\begin{align*}
& I:=\int_{B(x_0,R)}\int_0^{R/2}\varepsilon_n(x,t)^2\frac{dt}{t}d\mu(x)\quad \text{and} \\
&  II:=\int_{B(x_0,R)}\int_{R/2}^{R}\varepsilon_n(x,t)^2\frac{dt}{t}d\mu(x). 
\end{align*}
Using the trivial bound on $\varepsilon_n(x,t)$, it follows that 
\begin{equation}\label{e: II term bound}
    II\lesssim \mu(B(x_0,R)).
\end{equation}
We now bound $I$. From Lemma \ref{lem: eps controlled by beta},
 \begin{align}\label{e:eps and beta}
      &\int_0^{R/2} \varepsilon_n(x,t)^2\frac{dt}{t}
      \lesssim \sum_{k=0}^{\infty}\inf_{t\in \left(\frac{5}{4}\cdot2^{-(k+1)}R,2^{-k}R\right)}\obeta_{\mu,2}(x,t)^2
      \lesssim \int_0^{R}\obeta_{\mu,2}(x,t)^2\frac{dt}{t}.
 \end{align}
Since $\mu$ is uniformly $n$-rectifiable, from \eqref{e:eps and beta} and Lemma \ref{lem: centered beta and UR} we have, 
\begin{equation*}
I\leq \int_{B(x_0,R)}\int_0^{R}\obeta_{\mu,2}(x,t)^2\frac{dt}{t}d\mu(x)\leq C\mu(B(x_0,R)).
\end{equation*}
The theorem follows.
\end{proof}
We now turn to the proof of Lemma \ref{lem: eps controlled by beta}. The proof of this lemma is quite geometric. It essentially relies on the following intuition. Let $H^+$ be a half-space such that $\partial H^+$ minimizes $\obeta_2$. Then, $H^+$ is a competitor for $\ve_n$, and on any spherical shell the measure of $H^+\cap S(x,t)\setminus \Omega^+$ is contained in ``strips'' on $S(x,t)$ determined by the equator, $\partial H^+\cap S(x,t)$, and a collection of points $z_i\in \partial \Omega^+$. Essentially, the mass of $H^+\cap S(x,t)\setminus \Omega^+$ is controlled by how far the points $z_i$ are from $\partial H$. Integrating over a range of scales, these distances can be controlled by $\obeta_2$.

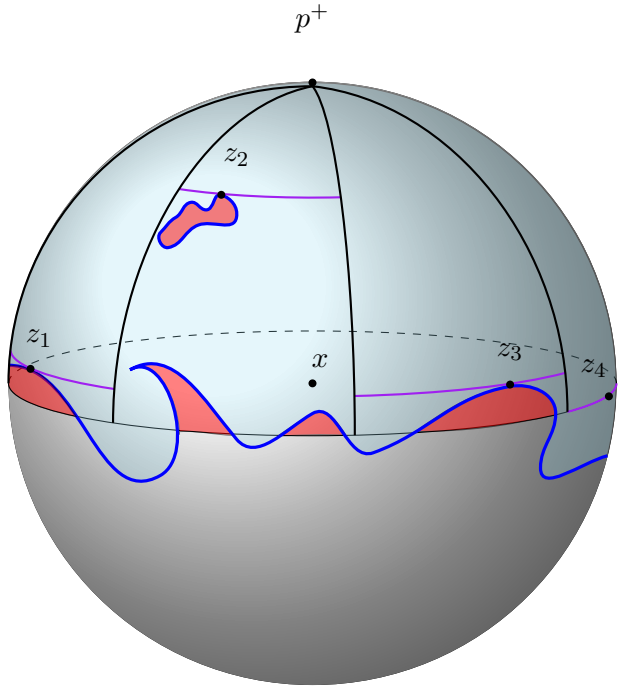
\begin{figure}[h]
    \centering
    \begin{tikzpicture}[scale=1,every node/.style={minimum size=1cm}]
    %%%%%% Code based on Tomas M. Trzeciak's `Stereographic and cylindrical map projections example`: 
%  http://www.texample.net/tikz/examples/map-projections/
% AND
% Marco Miani's work 'Spherical and cartesian grids example:
% https://texample.net/spherical-and-cartesian-grids/

%% some definitions

\def\R{4} % sphere radius

\def\angEl{10} % elevation angle
\def\angAz{-100} % azimuth angle
\def\angPhiOne{-110} % longitude of point P
\def\angPhiTwo{-45} % longitude of point Q
\def\angBeta{30} % latitude of point P and Q

%% working planes

\pgfmathsetmacro\H{\R*cos(\angEl)} % distance to north pole
\LongitudePlane[xzplane]{\angEl}{\angAz}
\LatitudePlane[equator]{\angEl}{0}
\fill[ball color=white!10] (0,0) circle (\R); % 3D lighting effect

\DrawLatitudeCircle[\R]{0} % equator

%%%%%%  drawing curve %%%%%%%

 \def\eps{1.2}
	\coordinate (D) at (-4,{.2*\eps});
	\def\d{0};
	\coordinate (E) at (-2,{-.6*\eps-.5});
	\def\e{30};
	\coordinate (F) at (-2.4,{.4*\eps-.3});
	\def\f{210};
    %%%% equator points for front cusp (F) plus ...
    \coordinate (G0) at (-1.75,{-.15*\eps-.5});
	  \def\gz{91};
    \coordinate (G1) at (-1.04,{-.17*\eps-.5});
 \def\ge{310};

 %\foreach \x /\y in {F/\f,G0/\gz,G1/\ge} {\tangy{\x}{\y}\labely{\x}{\y}} %checks coordinate/angle for each point

	\coordinate (G) at (-1,{-.2*\eps-.5});
	\def\g{-50};
    
	\coordinate (H) at (0,{0*\eps-.4});
	\def\h{30};
    %%%%%% equator points for bump including (H)
    \coordinate (H1) at (-.46,{-0.25*\eps-.4});
	\def\ha{35};
    \coordinate (H2) at (.44,{-0.24*\eps-.4});
	\def\hz{122};

 % \foreach \x /\y in {H/\h,H1/\ha,H2/\hz} {\tangy{\x}{\y}\labely{\x}{\y}} %checks coordinate/angle for each point
    
	\coordinate (I) at (.8,{-.1*\eps-.8});
	\def\i{20};
    
	\coordinate (J) at (2.6,{.3*\eps-.4});
	\def\j{10};
	\coordinate (K) at (3.2,{0*\eps-.4});
	\def\k{-90};
%%%%%%%% equator points for bump along with (J) and (K)
\coordinate (K1) at (1.35,{-0.2*\eps-.4});
	\def\ka{35};

% \foreach \x /\y in {J/\j,K/\k,K1/\ka} {\tangy{\x}{\y}\labely{\x}{\y}} %checks coordinate/angle for each point
    
    \coordinate (L) at (3,{-.6*\eps-.4});
	\def\l{-90};
	\coordinate (M) at (3.88,{-.8*\eps});
	\def\m{15};
    
 \def\addsm{.9}
	\coordinate (T) at (-2,{-.6*\eps+\addsm});
	\def\t{40};
    
	\coordinate (U) at (-2.4,{.4*\eps+\addsm});
	\def\u{-50};
%%%%%%%% equator points for back cusp with (U)
	\coordinate (U1) at (-1.93,{-.25*\eps+\addsm});
	\def\ua{105};
    \coordinate (U2) at (-1.85,{-.24*\eps+\addsm});
	\def\uz{121};

% \foreach \x /\y in {U/\u,U1/\ua,U2/\uz} {\tangy{\x}{\y}\labely{\x}{\y}} %checks coordinate/angle for each point

	\coordinate (V) at (0,{0*\eps-.5});
	\def\v{30};
    
	\coordinate (W) at (0,{0*\eps+\addsm});
	\def\w{30};
%%%%%%%%%% equator points for back pocket 1 with (W)
	\coordinate (W1) at (-0.16,{-0.18*\eps+\addsm});
	\def\wa{80};
    \coordinate (W2) at (0.6,{-0.18*\eps+\addsm});
	\def\wz{143};
% \foreach \x /\y in {W1/\wa,W/\w,W2/\wz} {\tangy{\x}{\y}\labely{\x}{\y}} %checks coordinate/angle for each point
    
	\coordinate (X) at (.8,{-.1*\eps+\addsm-.1});
	\def\x{20};
	\coordinate (Y) at (2.6,{.3*\eps+\addsm});
	\def\y{10};
%%%%%%%% equator points for back pocket 2 with (Y) and (Z)
\coordinate (X1) at (.85,{-.1*\eps+\addsm-.1});
	\def\xa{22};
\coordinate (Y1) at (3.23,{-0.4*\eps+\addsm});
	\def\ya{95};

% \foreach \x /\y in {X1/\xa,Y/\yy,Y1/\ya} {\tangy{\x}{\y}\labely{\x}{\y}} %checks coordinate/angle for each point
    
	\coordinate (Z) at (3.2,{0*\eps+\addsm});
	\def\z{-90};

    \coordinate (A) at (-2,2);
	\def\a{45};
	\coordinate (B) at (-1.8,2.2);
	\def\b{30};
	\coordinate (C) at (-1.6,2.4);
	\def\c{10};
	\coordinate (N1) at (-1.4,2.3);
	\def\n{0};
    
	\coordinate (O1) at (-1,2.2);
	\def\o{-100};
	\coordinate (P) at (-1.5,2.1);
	\def\p{200};
    \coordinate (R) at (-1.9,1.8);
	\def\r{190};
    \coordinate (S1) at (-1.2,2.5);
	\def\s{-20};

\filldraw[nearly transparent, fill=frontshade] %%%transparent shading for front
    (M) arc (-13.7:175:4) to
	(D) to [out=\d,in={\e-180}] 
	(E) to [out=\e,in={\f-180}] 
	(F) to [out=\f-180,in={\g-180}] 
	(G) to [out=\g,in={\h-180}] 
	(H) to [out=\h,in={\i-180}] 
	(I) to [out=\i,in={\j-180}] 
	(J) to [out=\j,in={\k-180}] 
	(K) to [out=\k,in={\l-180}] 
	(L) to [out=\l,in={\m-180}] 
    (M) -- cycle;

\filldraw[semitransparent, fill=red] %draws red front cusp 
   (-1.75,-.68) to [out={\gz}, in={\f-180}]
   (-2.4,.18) to  [out={\f-180}, in={\ge-180}]
   (-1.04,-0.69) to [out={178}, in={3}] (-1.75,-.64);	

\filldraw[semitransparent, fill=red] %draws red front pocket 2 
   (-.46,{-0.25*\eps-.4}) to [out={\ha}, in={\h-180}]
   (0,{0*\eps-.4}) to  [out={\h}, in={\hz}]
   (.44,{-0.24*\eps-.4}) to [out={178}, in={3}] (-.46,{-0.25*\eps-.4});	

\filldraw[semitransparent, fill=red] %draws red front pocket 3 
   (1.35,{-0.2*\eps-.4}) to [out={\ka}, in={\j-180}]
   (2.6,{.3*\eps-.4}) to  [out={\j}, in={\k-180}]
   (3.2,{0*\eps-.42}) to [out={192}, in={0.5}] (1.35,{-0.2*\eps-.4});

 \draw[very thick,blue] %draws front edge
	(D) to [out=\d,in={\e-180}] 
	(E) to [out=\e,in={\f-180}] 
	(F) to [out=\f-180,in={\g-180}] 
	(G) to [out=\g,in={\h-180}] 
	(H) to [out=\h,in={\i-180}] 
	(I) to [out=\i,in={\j-180}] 
	(J) to [out=\j,in={\k-180}] 
	(K) to [out=\k,in={\l-180}] 
	(L) to [out=\l,in={\m-180}] 
    (M);

\filldraw[ semitransparent, fill=red] %shades hole
	(A) to [out=\a,in={\b-180}] 
	(B) to [out=\b,in={\c-180}] 
	(C) to [out=\c,in={\n-180}] 
	(N1) to [out=\n,in={\s-180}] 
    (S1) to [out=\s,in={\o-180}] 
	(O1) to [out=\o,in={\p-180}] 
	(P) to [out=\p,in={\r-180}] 
	(R) to [out=\r,in={\a-180}] (A);

\draw[very thick,blue] %draws hole outline
     (A) to [out=\a,in={\b-180}] 
	(B) to [out=\b,in={\c-180}] 
	(C) to [out=\c,in={\n-180}] 
	(N1) to [out=\n,in={\s-180}] 
    (S1) to [out=\s,in={\o-180}] 
	(O1) to [out=\o,in={\p-180}] 
	(P) to [out=\p,in={\r-180}] 
	(R) to [out=\r,in={\a-180}] (A);

%%%%%% Labeling center and poles 
\coordinate[mark coordinate] (O) at (0,0);
\coordinate[mark coordinate] (N) at (0,{\H+.05});
\coordinate (S) at (0,{-\H-.05});

\node[above=8pt] at (N) {$p^+$};
%\node[below=8pt] at (S) {$\mathbf{p^-}$};

    %\draw[-,dashed, thick] (N) -- (S);
\node at (0.1,.3) {$x$};

\filldraw[ semitransparent, fill=red] %(D) arc (179:182.7:4) 
(-4,0) arc (180:219: 4 cm and .7 cm) to[out=125, in=\d] (D) --cycle;%draws ellipse with upper half dashed; inputs = {color}{circle center}{radius of circle in cm}{ellipse height in cm}{line thickness}

%%%%%% lattitude circle on bad points %%%%%%%%%%
\DrawPartialLatitudeCircle[\R]{48}{98-180}{49-180}
\DrawPartialLatitudeCircle[\R]{6.5}{49-180}{\angVis-180}
\DrawPartialLatitudeCircle[\R]{7.5}{147-180}{98-180}
\DrawPartialLatitudeCircle[\R]{0}{-\angVis}{147-180}

%%%%%%%%% longitude arcs %%%%%%%%%
\DrawLongitudeCirclearc[\R]{0}
\DrawLongitudeCirclearc[\R]{49}
\DrawLongitudeCirclearc[\R]{98}
\DrawLongitudeCirclearc[\R]{147}

%%%%% z_i points
\coordinate[mark coordinate] (N2) at (-1.2,2.5);
\coordinate (N7) at (-1,3);
\node at (N7) {$z_2$};
\coordinate[mark coordinate] (Z1) at (-3.71, {.16*\eps});
\coordinate (Z11) at (-3.6, {.5*\eps});
\node at (Z11) {$z_1$};
\coordinate[mark coordinate] (Z3) at (2.6,{.32*\eps-.4});
\coordinate (Z33) at (2.6,{.68*\eps-.4});
\node at (Z33) {$z_3$};
\coordinate[mark coordinate] (Z4) at (3.9,{.19*\eps-.4});
\coordinate (Z5) at (3.7,{.5*\eps-.4});
\node at (Z5) {$z_4$};

\end{tikzpicture}

    \caption{The region $H^+(t)\setminus \Omega^+$ is contained between the equator and the latitude line passing through ``bad'' point $z_i\in \partial \Omega$. The region on the equator between any two of the partial great circles is an $(n-1)$-ball of radius $\approx \frac{t}{N}$.}
    \label{fig:idea of geo proof}
\end{figure}

\begin{proof}[Proof of Lemma \ref{lem: eps controlled by beta}]
Let $\tau \in (0,1)$ be a small parameter to be fixed below (it will be a universal constant). Let $x\in \spt(\mu)$ and $0<r<\diam \spt(\mu)$. Fix $ s\in \left(\frac{5}{4}r,2r\right)$. If $\obeta_{\mu,2}(x,s)\geq \tau$, then it follows immediately that
\begin{equation*}
    \int_{r/2}^r\varepsilon_n(x,t)^2\frac{dt}{t}\lesssim_{\tau} \obeta_{\mu,2}(x,s)^2.
\end{equation*}
So, suppose that $\obeta_{\mu,2}(x,s)<\tau $, and let $H$ be the half-space such that $\partial H$ minimizes $\obeta_{\mu,2}(x,s)$. By rotating and translating, assume $\partial H=\{x_{n+1}=0\}$ and $H=\{x_{n+1}>0\}$. Let $H^+(t):=S(x,t)\cap H$ and let $H^-(t):=S(x,t)\cap H^-$, where $H^-=\mathbb{R}^{n+1}\setminus \overline{H}$. Note that $x$ is fixed throughout the proof, so we omit the dependence on $x$ from our notation. 

We first show that for any $N\geq 10$ and for all $t\in [r/2,r]$ there exists a finite collection of points $\{z_i\}_{1\leq i\leq N^{n-1}}$ in $H^+(t)\setminus \Omega^+$ such that
\begin{equation}\label{e: ptwse bound on epsn}
     \varepsilon_n(x,t)\lesssim \frac{1}{t^n}\sum_{i=1}^{N^{n-1}} \left(\frac{t}{N}\right)^{n-1}\dist(z_i,\partial H),
\end{equation}
and so that, moreover, the $(n+1)$-dimensional balls $U_{i}:=B(\zji, \frac{1}{4N} t)$ have finite overlap.
Note that the bound \eqref{e: ptwse bound on epsn} is for a fixed $t\in [r/2,r]$ and the points $z_i$ depend on $t$.

Fix $t\in [r/2,r]$. Observe that 
\begin{equation*}
    \varepsilon_n(x,t)\leq \frac{1}{t^n}\left(\varepsilon_n^+(x,t,H)+\varepsilon_n^-(x,t,H)\right), 
\end{equation*}
where 
\begin{equation*}
   \varepsilon_n^{+}(x,t,H)=\mathcal{H}^n(H^+(t)\setminus \Omega^+) \qquad \text{and}\qquad  \varepsilon_n^{-}(x,t,H)=\mathcal{H}^n(H^-(t)\setminus \Omega^-).
\end{equation*} 
Without loss of generality, we assume that $\varepsilon_n^{+}(x,t,H)\geq \varepsilon_n^{-}(x,t,H)$.

   \begin{figure}[h]
     \centering
\begin{minipage}[b]{0.4\textwidth}
\begin{tikzpicture}[scale=.75,every node/.style={minimum size=1cm}]
%%%%%% Code based on Tomas M. Trzeciak's `Stereographic and cylindrical map projections example`: 
%  http://www.texample.net/tikz/examples/map-projections/
% AND
% Marco Miani's work 'Spherical and cartesian grids example:
% https://texample.net/spherical-and-cartesian-grids/

%% some definitions

\def\R{4} % sphere radius

\def\angEl{10} % elevation angle
\def\angAz{-100} % azimuth angle
\def\angPhiOne{-110} % longitude of point P
\def\angPhiTwo{-45} % longitude of point Q
\def\angBeta{30} % latitude of point P and Q

%% working planes

\pgfmathsetmacro\H{\R*cos(\angEl)} % distance to north pole
\LongitudePlane[xzplane]{\angEl}{\angAz}
\LatitudePlane[equator]{\angEl}{0}
\fill[ball color=white!10] (0,0) circle (\R); % 3D lighting effect

%%%%%%  drawing curve %%%%%%%

 \def\eps{1.2}
	\coordinate (D) at (-4,{.2*\eps});
	\def\d{0};
	\coordinate (E) at (-2,{-.6*\eps-.5});
	\def\e{30};
	\coordinate (F) at (-2.4,{.4*\eps-.3});
	\def\f{210};
    %%%% equator points for front cusp (F) plus ...
    \coordinate (G0) at (-1.75,{-.15*\eps-.5});
	  \def\gz{91};
    \coordinate (G1) at (-1.04,{-.17*\eps-.5});
 \def\ge{310};

 %\foreach \x /\y in {F/\f,G0/\gz,G1/\ge} {\tangy{\x}{\y}\labely{\x}{\y}} %checks coordinate/angle for each point

	\coordinate (G) at (-1,{-.2*\eps-.5});
	\def\g{-50};
    
	\coordinate (H) at (0,{0*\eps-.4});
	\def\h{30};
    %%%%%% equator points for bump including (H)
    \coordinate (H1) at (-.46,{-0.25*\eps-.4});
	\def\ha{35};
    \coordinate (H2) at (.44,{-0.24*\eps-.4});
	\def\hz{122};

 % \foreach \x /\y in {H/\h,H1/\ha,H2/\hz} {\tangy{\x}{\y}\labely{\x}{\y}} %checks coordinate/angle for each point
    
	\coordinate (I) at (.8,{-.1*\eps-.8});
	\def\i{20};
    
	\coordinate (J) at (2.6,{.3*\eps-.4});
	\def\j{10};
	\coordinate (K) at (3.2,{0*\eps-.4});
	\def\k{-90};
%%%%%%%% equator points for bump along with (J) and (K)
\coordinate (K1) at (1.35,{-0.2*\eps-.4});
	\def\ka{35};

% \foreach \x /\y in {J/\j,K/\k,K1/\ka} {\tangy{\x}{\y}\labely{\x}{\y}} %checks coordinate/angle for each point
    
    \coordinate (L) at (3,{-.6*\eps-.4});
	\def\l{-90};
	\coordinate (M) at (3.88,{-.8*\eps});
	\def\m{15};
    
 \def\addsm{.9}
	\coordinate (T) at (-2,{-.6*\eps+\addsm});
	\def\t{40};
    
	\coordinate (U) at (-2.4,{.4*\eps+\addsm});
	\def\u{-50};
%%%%%%%% equator points for back cusp with (U)
	\coordinate (U1) at (-1.93,{-.25*\eps+\addsm});
	\def\ua{105};
    \coordinate (U2) at (-1.85,{-.24*\eps+\addsm});
	\def\uz{121};

% \foreach \x /\y in {U/\u,U1/\ua,U2/\uz} {\tangy{\x}{\y}\labely{\x}{\y}} %checks coordinate/angle for each point

	\coordinate (V) at (0,{0*\eps-.5});
	\def\v{30};
    
	\coordinate (W) at (0,{0*\eps+\addsm});
	\def\w{30};
%%%%%%%%%% equator points for back pocket 1 with (W)
	\coordinate (W1) at (-0.16,{-0.18*\eps+\addsm});
	\def\wa{80};
    \coordinate (W2) at (0.6,{-0.18*\eps+\addsm});
	\def\wz{143};
% \foreach \x /\y in {W1/\wa,W/\w,W2/\wz} {\tangy{\x}{\y}\labely{\x}{\y}} %checks coordinate/angle for each point
    
	\coordinate (X) at (.8,{-.1*\eps+\addsm-.1});
	\def\x{20};
	\coordinate (Y) at (2.6,{.3*\eps+\addsm});
	\def\y{10};
%%%%%%%% equator points for back pocket 2 with (Y) and (Z)
\coordinate (X1) at (.85,{-.1*\eps+\addsm-.1});
	\def\xa{22};
\coordinate (Y1) at (3.23,{-0.4*\eps+\addsm});
	\def\ya{95};

% \foreach \x /\y in {X1/\xa,Y/\yy,Y1/\ya} {\tangy{\x}{\y}\labely{\x}{\y}} %checks coordinate/angle for each point
    
	\coordinate (Z) at (3.2,{0*\eps+\addsm});
	\def\z{-90};

    \coordinate (A) at (-2,2);
	\def\a{45};
	\coordinate (B) at (-1.8,2.2);
	\def\b{30};
	\coordinate (C) at (-1.6,2.4);
	\def\c{10};
	\coordinate (N1) at (-1.4,2.3);
	\def\n{0};
    
	\coordinate (O1) at (-1,2.2);
	\def\o{-100};
	\coordinate (P) at (-1.5,2.1);
	\def\p{200};
    \coordinate (R) at (-1.9,1.8);
	\def\r{190};
    \coordinate (S1) at (-1.2,2.5);
	\def\s{-20};

\filldraw[nearly transparent, fill=frontshade] %%%transparent shading for front
    (M) arc (-13.7:175:4) to
	(D) to [out=\d,in={\e-180}] 
	(E) to [out=\e,in={\f-180}] 
	(F) to [out=\f-180,in={\g-180}] 
	(G) to [out=\g,in={\h-180}] 
	(H) to [out=\h,in={\i-180}] 
	(I) to [out=\i,in={\j-180}] 
	(J) to [out=\j,in={\k-180}] 
	(K) to [out=\k,in={\l-180}] 
	(L) to [out=\l,in={\m-180}] 
    (M) -- cycle;

\filldraw[semitransparent, fill=red] %draws red front cusp 
   (-1.75,-.68) to [out={\gz}, in={\f-180}]
   (-2.4,.18) to  [out={\f-180}, in={\ge-180}]
   (-1.04,-0.69) to [out={178}, in={3}] (-1.75,-.64);	

\filldraw[semitransparent, fill=red] %draws red front pocket 2 
   (-.46,{-0.25*\eps-.4}) to [out={\ha}, in={\h-180}]
   (0,{0*\eps-.4}) to  [out={\h}, in={\hz}]
   (.44,{-0.24*\eps-.4}) to [out={178}, in={3}] (-.46,{-0.25*\eps-.4});	

\filldraw[semitransparent, fill=red] %draws red front pocket 3 
   (1.35,{-0.2*\eps-.4}) to [out={\ka}, in={\j-180}]
   (2.6,{.3*\eps-.4}) to  [out={\j}, in={\k-180}]
   (3.2,{0*\eps-.42}) to [out={192}, in={0.5}] (1.35,{-0.2*\eps-.4});

 \draw[very thick,blue] %draws front edge
	(D) to [out=\d,in={\e-180}] 
	(E) to [out=\e,in={\f-180}] 
	(F) to [out=\f-180,in={\g-180}] 
	(G) to [out=\g,in={\h-180}] 
	(H) to [out=\h,in={\i-180}] 
	(I) to [out=\i,in={\j-180}] 
	(J) to [out=\j,in={\k-180}] 
	(K) to [out=\k,in={\l-180}] 
	(L) to [out=\l,in={\m-180}] 
    (M);

\filldraw[ semitransparent, fill=red] %shades hole
	(A) to [out=\a,in={\b-180}] 
	(B) to [out=\b,in={\c-180}] 
	(C) to [out=\c,in={\n-180}] 
	(N1) to [out=\n,in={\s-180}] 
    (S1) to [out=\s,in={\o-180}] 
	(O1) to [out=\o,in={\p-180}] 
	(P) to [out=\p,in={\r-180}] 
	(R) to [out=\r,in={\a-180}] (A);

\draw[very thick,blue] %draws hole outline
     (A) to [out=\a,in={\b-180}] 
	(B) to [out=\b,in={\c-180}] 
	(C) to [out=\c,in={\n-180}] 
	(N1) to [out=\n,in={\s-180}] 
    (S1) to [out=\s,in={\o-180}] 
	(O1) to [out=\o,in={\p-180}] 
	(P) to [out=\p,in={\r-180}] 
	(R) to [out=\r,in={\a-180}] (A);

%%%%%% Labeling center and poles 
\coordinate[mark coordinate] (O) at (0,0);
\coordinate[mark coordinate] (N) at (0,{\H+.05});
\coordinate (S) at (0,{-\H-.05});

\node[above=8pt] at (N) {$p^+$};
%\node[below=8pt] at (S) {$\mathbf{p^-}$};

    %\draw[-,dashed, thick] (N) -- (S);
\node at (0.1,.3) {$x$};

\filldraw[ semitransparent, fill=red] %(D) arc (179:182.7:4) 
(-4,0) arc (180:219: 4 cm and .7 cm) to[out=125, in=\d] (D) --cycle;%draws ellipse with upper half dashed; inputs = {color}{circle center}{radius of circle in cm}{ellipse height in cm}{line thickness}

%%%%%%%%% longitude arcs %%%%%%%%%
\DrawLongitudeCirclearc[\R]{0}
\DrawLongitudeCirclearc[\R]{49}
\DrawLongitudeCirclearc[\R]{98}
\DrawLongitudeCirclearc[\R]{147}

%%%%% Draws front side equator %%%%%%%%%%%%
\DrawPartialLatitudeCirclethick[\R]{0}{-\angVis}{\angVis-180} % equator

%%%%%% Balls on the equator green %%%
\coordinate[mark coordinate] (Eq1) at (-4,{0.3*\eps-.4});
\coordinate[mark coordinate] (Eq2) at (-2.63,{-0.1*\eps-.4});
\coordinate[mark coordinate] (Eq3) at (0.56,{-0.23*\eps-.4});
\coordinate[mark coordinate] (Eq4) at (3.37,{0.01*\eps-.4});

%%%%%% Draws an example of $A^+(\theta)$
\DrawLongitudeCirclearccolor[\R]{62.5}
\coordinate (N33) at (-1.4,1.5);
\node[right] at (N33) {{\color{purple}$A^+(\theta_i)$}};
\coordinate[mark purple coordinate] (Eq5) at (-1.85, {-0.18*\eps-.4});
\coordinate (Eq6) at (-2.1, {-0.31*\eps-.4});
\node at (Eq6) {$\theta_i$};
\end{tikzpicture}
\end{minipage}
\hfill 
\begin{minipage}[b]{0.4\textwidth}
\begin{tikzpicture}[scale=0.75,every node/.style={minimum size=1cm}]
%%%%%% Code based on Tomas M. Trzeciak's `Stereographic and cylindrical map projections example`: 
%  http://www.texample.net/tikz/examples/map-projections/
% AND
% Marco Miani's work 'Spherical and cartesian grids example:
% https://texample.net/spherical-and-cartesian-grids/

%% some definitions

\def\R{4} % sphere radius

\def\angEl{10} % elevation angle
\def\angAz{-100} % azimuth angle
\def\angPhiOne{-110} % longitude of point P
\def\angPhiTwo{-45} % longitude of point Q
\def\angBeta{30} % latitude of point P and Q

%% working planes

\pgfmathsetmacro\H{\R*cos(\angEl)} % distance to north pole
\LongitudePlane[xzplane]{\angEl}{\angAz}
\LatitudePlane[equator]{\angEl}{0}
\fill[ball color=white!10] (0,0) circle (\R); % 3D lighting effect

%%%%%%  drawing curve %%%%%%%

 \def\eps{1.2}
	\coordinate (D) at (-4,{.2*\eps});
	\def\d{0};
	\coordinate (E) at (-2,{-.6*\eps-.5});
	\def\e{30};
	\coordinate (F) at (-2.4,{.4*\eps-.3});
	\def\f{210};
    %%%% equator points for front cusp (F) plus ...
    \coordinate (G0) at (-1.75,{-.15*\eps-.5});
	  \def\gz{91};
    \coordinate (G1) at (-1.04,{-.17*\eps-.5});
 \def\ge{310};

 %\foreach \x /\y in {F/\f,G0/\gz,G1/\ge} {\tangy{\x}{\y}\labely{\x}{\y}} %checks coordinate/angle for each point

	\coordinate (G) at (-1,{-.2*\eps-.5});
	\def\g{-50};
    
	\coordinate (H) at (0,{0*\eps-.4});
	\def\h{30};
    %%%%%% equator points for bump including (H)
    \coordinate (H1) at (-.46,{-0.25*\eps-.4});
	\def\ha{35};
    \coordinate (H2) at (.44,{-0.24*\eps-.4});
	\def\hz{122};

 % \foreach \x /\y in {H/\h,H1/\ha,H2/\hz} {\tangy{\x}{\y}\labely{\x}{\y}} %checks coordinate/angle for each point
    
	\coordinate (I) at (.8,{-.1*\eps-.8});
	\def\i{20};
    
	\coordinate (J) at (2.6,{.3*\eps-.4});
	\def\j{10};
	\coordinate (K) at (3.2,{0*\eps-.4});
	\def\k{-90};
%%%%%%%% equator points for bump along with (J) and (K)
\coordinate (K1) at (1.35,{-0.2*\eps-.4});
	\def\ka{35};

% \foreach \x /\y in {J/\j,K/\k,K1/\ka} {\tangy{\x}{\y}\labely{\x}{\y}} %checks coordinate/angle for each point
    
    \coordinate (L) at (3,{-.6*\eps-.4});
	\def\l{-90};
	\coordinate (M) at (3.88,{-.8*\eps});
	\def\m{15};
    
 \def\addsm{.9}
	\coordinate (T) at (-2,{-.6*\eps+\addsm});
	\def\t{40};
    
	\coordinate (U) at (-2.4,{.4*\eps+\addsm});
	\def\u{-50};
%%%%%%%% equator points for back cusp with (U)
	\coordinate (U1) at (-1.93,{-.25*\eps+\addsm});
	\def\ua{105};
    \coordinate (U2) at (-1.85,{-.24*\eps+\addsm});
	\def\uz{121};

% \foreach \x /\y in {U/\u,U1/\ua,U2/\uz} {\tangy{\x}{\y}\labely{\x}{\y}} %checks coordinate/angle for each point

	\coordinate (V) at (0,{0*\eps-.5});
	\def\v{30};
    
	\coordinate (W) at (0,{0*\eps+\addsm});
	\def\w{30};
%%%%%%%%%% equator points for back pocket 1 with (W)
	\coordinate (W1) at (-0.16,{-0.18*\eps+\addsm});
	\def\wa{80};
    \coordinate (W2) at (0.6,{-0.18*\eps+\addsm});
	\def\wz{143};
% \foreach \x /\y in {W1/\wa,W/\w,W2/\wz} {\tangy{\x}{\y}\labely{\x}{\y}} %checks coordinate/angle for each point
    
	\coordinate (X) at (.8,{-.1*\eps+\addsm-.1});
	\def\x{20};
	\coordinate (Y) at (2.6,{.3*\eps+\addsm});
	\def\y{10};
%%%%%%%% equator points for back pocket 2 with (Y) and (Z)
\coordinate (X1) at (.85,{-.1*\eps+\addsm-.1});
	\def\xa{22};
\coordinate (Y1) at (3.23,{-0.4*\eps+\addsm});
	\def\ya{95};

% \foreach \x /\y in {X1/\xa,Y/\yy,Y1/\ya} {\tangy{\x}{\y}\labely{\x}{\y}} %checks coordinate/angle for each point
    
	\coordinate (Z) at (3.2,{0*\eps+\addsm});
	\def\z{-90};

    \coordinate (A) at (-2,2);
	\def\a{45};
	\coordinate (B) at (-1.8,2.2);
	\def\b{30};
	\coordinate (C) at (-1.6,2.4);
	\def\c{10};
	\coordinate (N1) at (-1.4,2.3);
	\def\n{0};
    
	\coordinate (O1) at (-1,2.2);
	\def\o{-100};
	\coordinate (P) at (-1.5,2.1);
	\def\p{200};
    \coordinate (R) at (-1.9,1.8);
	\def\r{190};
    \coordinate (S1) at (-1.2,2.5);
	\def\s{-20};

\filldraw[nearly transparent, fill=frontshade] %%%transparent shading for front
    (M) arc (-13.7:175:4) to
	(D) to [out=\d,in={\e-180}] 
	(E) to [out=\e,in={\f-180}] 
	(F) to [out=\f-180,in={\g-180}] 
	(G) to [out=\g,in={\h-180}] 
	(H) to [out=\h,in={\i-180}] 
	(I) to [out=\i,in={\j-180}] 
	(J) to [out=\j,in={\k-180}] 
	(K) to [out=\k,in={\l-180}] 
	(L) to [out=\l,in={\m-180}] 
    (M) -- cycle;

\filldraw[semitransparent, fill=red] %draws red front cusp 
   (-1.75,-.68) to [out={\gz}, in={\f-180}]
   (-2.4,.18) to  [out={\f-180}, in={\ge-180}]
   (-1.04,-0.69) to [out={178}, in={3}] (-1.75,-.64);	

\filldraw[semitransparent, fill=red] %draws red front pocket 2 
   (-.46,{-0.25*\eps-.4}) to [out={\ha}, in={\h-180}]
   (0,{0*\eps-.4}) to  [out={\h}, in={\hz}]
   (.44,{-0.24*\eps-.4}) to [out={178}, in={3}] (-.46,{-0.25*\eps-.4});	

\filldraw[semitransparent, fill=red] %draws red front pocket 3 
   (1.35,{-0.2*\eps-.4}) to [out={\ka}, in={\j-180}]
   (2.6,{.3*\eps-.4}) to  [out={\j}, in={\k-180}]
   (3.2,{0*\eps-.42}) to [out={192}, in={0.5}] (1.35,{-0.2*\eps-.4});

 \draw[very thick,blue] %draws front edge
	(D) to [out=\d,in={\e-180}] 
	(E) to [out=\e,in={\f-180}] 
	(F) to [out=\f-180,in={\g-180}] 
	(G) to [out=\g,in={\h-180}] 
	(H) to [out=\h,in={\i-180}] 
	(I) to [out=\i,in={\j-180}] 
	(J) to [out=\j,in={\k-180}] 
	(K) to [out=\k,in={\l-180}] 
	(L) to [out=\l,in={\m-180}] 
    (M);

\filldraw[ semitransparent, fill=red] %shades hole
	(A) to [out=\a,in={\b-180}] 
	(B) to [out=\b,in={\c-180}] 
	(C) to [out=\c,in={\n-180}] 
	(N1) to [out=\n,in={\s-180}] 
    (S1) to [out=\s,in={\o-180}] 
	(O1) to [out=\o,in={\p-180}] 
	(P) to [out=\p,in={\r-180}] 
	(R) to [out=\r,in={\a-180}] (A);

\draw[very thick,blue] %draws hole outline
     (A) to [out=\a,in={\b-180}] 
	(B) to [out=\b,in={\c-180}] 
	(C) to [out=\c,in={\n-180}] 
	(N1) to [out=\n,in={\s-180}] 
    (S1) to [out=\s,in={\o-180}] 
	(O1) to [out=\o,in={\p-180}] 
	(P) to [out=\p,in={\r-180}] 
	(R) to [out=\r,in={\a-180}] (A);

%%%%%% Labeling center and poles 
\coordinate[mark coordinate] (O) at (0,0);
\coordinate[mark coordinate] (N) at (0,{\H+.05});
\coordinate (S) at (0,{-\H-.05});

\node[above=8pt] at (N) {$p^+$};
%\node[below=8pt] at (S) {$\mathbf{p^-}$};

    %\draw[-,dashed, thick] (N) -- (S);
\node at (0.1,.3) {$x$};

\filldraw[ semitransparent, fill=red] %(D) arc (179:182.7:4) 
(-4,0) arc (180:219: 4 cm and .7 cm) to[out=125, in=\d] (D) --cycle;%draws ellipse with upper half dashed; inputs = {color}{circle center}{radius of circle in cm}{ellipse height in cm}{line thickness}

%%%%%%%%% longitude arcs %%%%%%%%%
\DrawLongitudeCirclearc[\R]{0}
\DrawLongitudeCirclearc[\R]{49}
\DrawLongitudeCirclearc[\R]{98}
\DrawLongitudeCirclearc[\R]{147}

%%%%% Draws front side equator %%%%%%%%%%%%
\DrawPartialLatitudeCirclethick[\R]{0}{-\angVis}{\angVis-180} % equator

%%%%%% Balls on the equator green %%%
\coordinate[mark coordinate] (Eq1) at (-4,{0.3*\eps-.4});
\coordinate[mark coordinate] (Eq2) at (-2.63,{-0.1*\eps-.4});
\coordinate[mark coordinate] (Eq3) at (0.56,{-0.23*\eps-.4});
\coordinate[mark coordinate] (Eq4) at (3.37,{0.01*\eps-.4});

%%%%%% Draws an example of $A^+(\theta)$
\DrawLongitudeCirclearccolor[\R]{62.5}
\coordinate (N33) at (-1.3,1.5);
%\node[right] at (N33) {{\color{purple}$A^+(\theta_i)$}};
\coordinate[mark purple coordinate] (Eq5) at (-1.85, {-0.18*\eps-.4});
\coordinate (Eq6) at (-2.1, {-0.31*\eps-.4});
\node at (Eq6) {$\theta_i$};

\coordinate[mark coordinate] (N2) at (-1.24,2.5);
\coordinate (N7) at (-1.43,2.75);
\node at (N7) {$z_i$};

%%%%%%%% Draw C^+(\theta)
\DrawLongitudeCirclearctoz[\R]{62.5}
\coordinate (N4) at (-0.35,2.9);
\node at (N4) {\small{$C^{+}(\theta_i)$}};
\end{tikzpicture}
\end{minipage}
\caption{The arc $A^+(\theta_i)$ and the subarc $C^+(\theta_i)$.}
\end{figure}
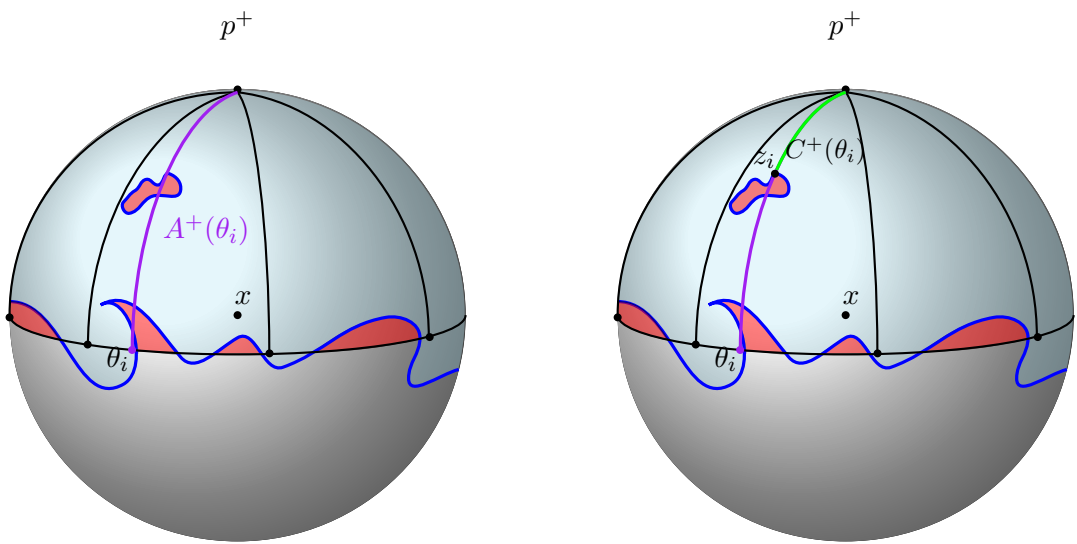

Let us first set some notation. Let $p^+_H=(0,\hdots,t)$ denote the north pole of $H^+(t)$ and for $\theta \in S(x,t)\cap \partial H$ denote the minimal arc on $S(x,t)$ between $p^+_H$ and $\theta$ by $A^{+}(\theta)$. Observe that $A^+(\theta)\subset H^+(t)$.\\

\begin{claim}\label{claim-ven-arcs}
\begin{align}\label{e: Jacobian}
    \epshplus(x,t,H)
    \leq\int_{S(x,t)\cap \partial H}\mathcal{H}^1\left(A^+(\theta)\setminus \Omega^+\right)d\theta. 
\end{align}
\end{claim}

\vspace{0.5cm}

Let us continue with the proof of the lemma assuming the claim to be true. We will go back to its demonstration in due time.

Denote $\delta=\frac{1}{2N}$ and let $\{\Delta_{i}\}_{i=1}^{N^{n-1}}$ 
a family of $(n-1)$-dimensional balls in $S(x,t)\cap \partial H$ with $\rad(\Delta_{i})\approx \delta t$ such that
 $\sum_{i=1}^{N^{n-1}}\chi_{3\Delta_{i}}\leq C$, for some absolute constant $C>0$. Then we choose
$\theta_{i}\in \Delta_{i}$ so that 
\begin{equation*}
    \frac{1}{2}\sup_{\theta\in \Delta_{i}}\mathcal{H}^1\left(A^+(\theta)\setminus \Omega^+\right)\leq \mathcal{H}^1\left(A^+(\theta_{i})\setminus \Omega^+\right),
\end{equation*}
and then, applying Claim \ref{claim-ven-arcs}, we see that 
\begin{align*}
     \epshplus(x,t,H)
     &\leq \sum_{i=1}^{N^{n-1}} \int_{\Delta_{i}}\mathcal{H}^1\left(A^+(\theta)\setminus \Omega^+\right)d\theta
     \lesssim \sum_{i=1}^{N^{n-1}} (\delta t)^{n-1}\mathcal{H}^1\left(A^+(\theta_{i})\setminus \Omega^+\right), 
\end{align*}
In order to estimate $\mathcal{H}^1\left(A^+(\theta_{i})\setminus \Omega^+\right)$, 
we choose $z_i\in \left( A^+(\theta_{i})\setminus \Omega^+\right)\cap \partial \Omega^+$ so that 
\begin{equation*}
    \frac{1}{2}\sup_{z\in A^+(\theta)\setminus \Omega^+} \dist(z,\partial H)\leq d(z_{i},\partial H).
\end{equation*}
In the case that $A^+(\theta)\setminus \Omega^+=\varnothing$, we take $\zji=\theta_{i}$.
Now we define $C^+(\theta_{i})$ to be the sub-arc of $A^+(\theta_{i})$ with endpoints $p_H^+$ and $\zji$. From the $n$-ADR of $\mu$ and the assumption that $\obeta_{\mu,2}(x,s)< \tau$, it easily follows that $C^+(\theta_{i})\neq \varnothing$, whenever $\tau \in (0,1)$ is chosen sufficiently small. 
Because of the same reason, we may assume that $|z_i-p_H^+|\geq \frac1{20}r$.
Then we have
\begin{equation}\label{e: epsnplus bounded by worst points}
    \epshplus(x,t,H)\lesssim \sum_{i=1}^{N^{n-1}} (\delta t)^{n-1}\left|\frac{\pi}{2}t-\mathcal{H}^1(C^+(\theta_{i})) \right|\lesssim \sum_{i=1}^{N^{n-1}} (\delta t)^{n-1}\dist(\zji,\partial H), 
\end{equation}
which proves \eqref{e: ptwse bound on epsn}. 
Also, since  the points $z_i$ are far away from the north pole and $\sum_{i=1}^{N^{n-1}}\chi_{3\Delta_i}\leq C$, 
we also have $\sum_{i=1}^{N^{n-1}}\chi_{U_i}\leq C'$, for some absolute constant $C'$.

Define 
\begin{equation*}
    \cB_{\delta}=\{i\in [1,N^{n-1}]: \, \dist(\zji,\partial H)\geq \delta t\},
\end{equation*}
then, 
\begin{align}\label{e: first sum bound epshplus}
\begin{split}
   \sum_{i=1}^{N^{n-1}} (\delta t)^{n-1}\dist(\zji,\partial H)
    &\lesssim \delta t^n +(\delta t)^{n-1}\sum_{i\in \cB_{\delta}} \dist(\zji,\partial H).
    \end{split}
\end{align}
To estimate $\dist(\zji,\partial H)$ for $i\in \cB_{\delta}$, consider the $(n+1)$-dimensional ball $U_{i}=B(\zji, \frac{1}{2}\delta t)$.  We have,  
\begin{align*}
    (\delta t)^{n-1}\sum_{i\in \cB_{\delta}} \dist(\zji,\partial H)
    &\lesssim\frac{1}{(\delta t)}\sum_{i\in \cB_{\delta}} \int_{U_{i}} \dist(y, \partial H)d\mu(y),
\end{align*}
by the Ahlfors regularity of $\mu$ and since $\inf_{y\in U_i}\dist(y,\partial H)\approx \dist(\zji, \partial H)$. %for any $y \in U_i=B(z_i, \frac12 \delta t)$. 
Now combining this with  \eqref{e: epsnplus bounded by worst points} and \eqref{e: first sum bound epshplus} yields 
\begin{equation}\label{e: epsn final}
  \varepsilon_n(x,t,H)\leq 2\,\epshplus(x,t,H) \lesssim \delta t^n +\frac{1}{(\delta t)}\sum_{i\in \cB_{\delta}} \int_{U_{i}} \dist(y, \partial H)d\mu(y). 
\end{equation}

We now complete the proof of \eqref{e: eps bounded by obeta}. For $N$ as above, partition $[r/2,r]$ into $N$ intervals, $\{J_j\}_{j=1}^N$, such that for each $j$,  $|J_j|=\delta r$ (recall that $\delta=\frac{1}{2N}$). Then, 
\begin{equation*}
    \int_{r/2}^r \varepsilon_n(x,t)^2 \frac{dt}{t}
    \lesssim \frac{1}{r}\sum_{j=1}^N \int_{J_j}\sup_{t\in J_j}\varepsilon_n(x,t)^2dt\lesssim \frac{1}{r}\sum_{j=1}^N \varepsilon_n(x,t_j)^2(\delta r), 
\end{equation*}
where $t_j\in J_j$ is chosen so that $\varepsilon_n(x,t_j)^2\geq \frac{1}{2}\sup_{t\in J_j}\varepsilon_n(x,t)^2$. Thus, unpacking definitions, we see that
\begin{equation*}\label{e: bound eps on shells}
    \int_{r/2}^r \varepsilon_n(x,t)^2 \frac{dt}{t}
    \lesssim \delta\sum_{j=1}^N \left(\frac{\max\{\varepsilon_n^+(x,t_j,H), \varepsilon_n^-(x,t_j,H)\}}{t_j^n}\right)^2.
\end{equation*}

Let $z_{j,i}^*$ denote the point $z_i$ as found above for $t=t_j$ and let $U_{j,i}= B(z_{j,i}^*, \frac12 \delta t_j)$ be the $(n+1)$-dimensional ball $U_i$ centered at $z_{j,i}^{*}$. Then set $\cB_{\delta}^j=\{i\in [1,N^{n-1}]: \, d(z_{j,i}^*,\partial H)\geq \delta t_j\}$. Now, from \eqref{e: epsn final} we have that
\begin{align*}
    \int_{r/2}^r \varepsilon_n(x,t)^2 \frac{dt}{t}
    &\lesssim \delta^2+\frac{1}{\delta}\sum_{j=1}^N \frac{1}{t_j^{2n}}\left(\sum_{i\in \cB_{\delta}^j} \int_{U_{j,i}} \frac{\dist(y, \partial H)}{t_j}d\mu(y)\right)^2\\
    &\leq \delta^2+\frac{1}{\delta}\sum_{j=1}^N \frac{1}{t_j^{2n}}\left(\sum_{i\in \cB_{\delta}^j} \int_{U_{j,i}} \left(\frac{\dist(y, \partial H)}{t_j}\right)^2d\mu(y)\right)\left(\sum_{i\in \cB_{\delta}^j}\mu(U_{j,i})\right), 
\end{align*}
where the second inequality follows from two applications of Cauchy-Schwarz. Since $\mu(U_{j,i})\approx (\delta t_j)^{n}$ and $\#\cB_{\delta}\leq N^{n-1}\approx\frac{1}{\delta^{n-1}}$, then for a fixed $1\leq j \leq N$,
\begin{equation*}
    \sum_{i \in \cB_\delta^j} \mu(U_{j,i}) \lesssim \delta t_j^n,
\end{equation*}
where $\#\cB_{\delta}$ denotes the cardinality of $\cB_{\delta}$. Thus, considering that $t_j \approx r$, we get 
\begin{equation*}
     \int_{r/2}^r \varepsilon_n(x,t)^2 \frac{dt}{t}
     \lesssim \delta^2+\frac{1}{r^n}\sum_{j=1}^N \sum_{i\in \cB_{\delta}^j} \int_{U_{j,i}} \left(\frac{\dist(y, \partial H)}{t_j}\right)^2d\mu(y).
\end{equation*} 
From the finite overlap of the balls $U_i$ and the fact that for $U_{j,i}\cap U_{k,\ell}=\varnothing$ for any $k\neq j-1,j+1$ and any $\ell$, we deduce that $\sum_{j=1}^N \sum_{i=1}^{N^{n-1}}\chi_{U_{j,i}}\leq 3C$. 
Hence, continuing the estimate from above gives 
\begin{align*}
     \int_{r/2}^r \varepsilon_n(x,t)^2 \frac{dt}{t}
     &\lesssim \delta^2+ \frac{1}{r^n} \int_{\bigcup_{i,j} U_{j,i}} \left(\frac{\dist(y, \partial H)}{r}\right)^2d\mu(y)\\
     &\lesssim \delta^2+ \obeta_{\mu,2}(x,s)^2.
\end{align*}
Letting $\delta\to 0$ we conclude the proof of Lemma \ref{lem: eps controlled by beta}. 
\end{proof} 

Our reader is still due a proof of Claim \ref{claim-ven-arcs}.

\begin{proof}[Proof of Claim \ref{claim-ven-arcs}]
 Let $\mathbb{S}_+^{n}(t):=S(x,t)\cap H$ and $E:= \overline{\Omega^-}$. Then, 
    \begin{equation}\label{e: Jacobian estimate 1}
        \int_{\mathbb{S}_+^n(t)} \chi_{E}(z)d\mathcal{H}^n\rest_{\mathbb{S}^n(t)}=\int_0^t\int_{\Gamma_s}\chi_E(z) d\sigma_{n-1}^s \, ds, 
    \end{equation}
    where $\Gamma_s=\mathbb{S}_+^n(t)\cap\{z_{n+1}=s\}$, 
    $r_s$ is the radius of $\Gamma_s$, and $\sigma_{n-1}^s=\mathcal{H}^{n-1}\rest_{\Gamma_s}$. For reference, see \cite[Theorem 3.11]{evans2015measure}. We consider the map  
    \begin{equation*}
        f_s: \Gamma_0\to \Gamma_s \qquad \text{such that}\qquad (w,0)\mapsto\left(\frac{r_s}{t}w, s\right).
    \end{equation*}
     We make the substitution $z=f(\theta,0)$ in \eqref{e: Jacobian estimate 1}, and then we get 
    \begin{equation}\label{e: Jacobian estimate 2}
        \int_{\mathbb{S}_+^n(t)} \chi_{E}(z)d\mathcal{H}^n\rest_{\mathbb{S}^n(t)}=\int_0^t\int_{\Gamma_0}\chi_E\circ f_s(\theta,0)\left(\frac{r_s}{t}\right)^{n-1}d\mathcal{H}^{n-1}\rest_{\Gamma_0}(\theta)ds. 
    \end{equation}
    Define $\gamma$ to be the angle measured from the $z_{n+1}$ positive semi-axis. Then, let 
    \begin{equation*}
        s=t\sin\left(\alpha\right), \qquad \text{where}\qquad \alpha=\frac{\pi}{2}-\gamma.
    \end{equation*}
  From \eqref{e: Jacobian estimate 2} and since $t\cos(\alpha)=r_s$, we have 
  \begin{align*}
       \int_{\mathbb{S}_+^n(t)} \chi_{E}(z)d\mathcal{H}^n\rest_{\mathbb{S}^n(t)}
       &= \int_0^{\pi/2}t\int_{\Gamma_0}\chi_E\circ f_{s(\alpha)}(\theta,0)\cos(\alpha)^{n}\,
       d\mathcal{H}^{n-1}\rest_{\Gamma_0}(\theta)d\alpha\\
       &\leq \int_{\Gamma_0}\int_0^{\pi/2}\chi_E(\theta\cos(\alpha),t\sin(\alpha)) \,t\,d\alpha \,d\mathcal{H}^{n-1}(\theta)\\
       &=\int_{\Gamma_0}\mathcal{H}^1(E\cap A^+(\theta))\,d\mathcal{H}^{n-1}(\theta).
  \end{align*}
  Recalling the definitions of $E$ and $\Gamma_0$, \eqref{e: Jacobian} holds.
\end{proof}

\section{Estimates on CAD domains}\label{s:CAD}

Up to now we have shown that $(2) \iff (1)$ in Theorem \ref{t:main2}, see Sections \ref{s:(2)implies(1)} and \ref{s:(1)implies(2)}. Additionally, $(3) \implies (1)$ follows from $(2) \implies (1)$ and Theorem \ref{teofac***}. The next two sections are devoted to the proof of $(1) \implies (3)$. 
More specifically, here we will show some estimates on chord-arc domains (CADs). They will be used in the next section to complete the proof of Theorem \ref{t:main2}. 

We begin by recalling two key lemmas from \cite{JK82}.
\begin{lemma}[Lemma 4.4, \cite{JK82}]\label{lem:bound on harmonic measure corkscrew point}
	Let $\Omega$ be an 
	NTA domain. Given a compact set $K\subset \mathbb{R}^{n+1}$ for $x\in \partial \Omega \cap K$ and $0<2r<R_K$. If $u\geq 0$ is a harmonic function in $\Omega\cap B(x,4r)$ and $u$ vanishes continuously on $B(x,2r)\cap \partial \Omega$ then 
	\begin{equation*}
		u(p)\leq Cu(q_{x,r}) \qquad \text{for all}\quad p\in B(x,r) \cap \Omega,
	\end{equation*} 
	where $C$ depends only on $K$ and $q_{x,r}$ is the corkscrew point for $x$ at scale $r$ in $\Omega$. 
\end{lemma}

\begin{lemma}[Lemma 4.8, \cite{JK82}]\label{lem:density of harmonic measure unif bounded}
	Let $\Omega$ be an 
	NTA domain. Given a compact set $K\subset \mathbb{R}^{n+1}$ for $x\in \partial \Omega \cap K$, $0<2r<R_K$ and $p\in \Omega\setminus B(x,2r)$. Then 
	\begin{equation*}
		C^{-1}<\frac{\omega^p(B(x,r))}{r^{n-1}G(q_{x,r},p)}<C,
	\end{equation*}
	where $G(\cdot,p)$ is the Green function of $\Omega$ with pole $p$ and $q_{x,r}$ is the corkscrew point for $x$ at scale $r$.
\end{lemma}

Next, the following notation will be useful.

\begin{definition}\label{def: density ratios}
	For $x\in \mathbb{R}^{n+1}$ and $r>0$ define the \textit{density ratio} of the harmonic measure $\omega_i$ of $\Omega_i$ with pole at $p_i$ as
	\begin{equation*}
		\theta_i(x,r):=\theta_{\omega_i}(x,r):=\frac{\omega_i(B(x,r))}{r^n}, \qquad \text{for}\quad i=1,2.
	\end{equation*}
\end{definition}

Finally, let us state the first of the two lemmas to be proven in this section. Note that it may be thought of as a quantitative version of \cite[Theorem D]{FTV24}.

\begin{lemma}\label{lem: bounding by logs CAD}
Let $c_1 >0$ and 
	let $\Omega_1,\Omega_2\subset \mathbb{R}^{n+1}$ be disjoint NTA domains. For  $i=1,2$, let $p_i\in \Omega_i$ be such that $\dist(p_i,\partial \Omega_i)\geq c_1\diam(\partial \Omega_i)$. Denote by $\omega_i$ the harmonic measure for $\Omega_i$ with respect to the pole $p_i$. Let $x\in \mathbb{R}^{n+1}\setminus (\Omega_1\cup \Omega_2)$ and denote $\delta_x=\max_i(\dist(x,\partial \Omega_i))$. For $\rho,r$ such that $2\delta_x\leq \rho\leq r\leq \frac{\min_i(\dist(x,p_i))}{4}$, we have
\begin{equation*}
		\int_{\rho}^{r}\frac{\alpha_1(x,t)+\alpha_2(x,t)-2}{t}\,dt\leq \log\left(\frac{\theta_{1}(x,r)}{\theta_1(x,\rho)}\right) +\log\left(\frac{\theta_{2}(x,r)}{\theta_2(x,\rho)}\right)+C,
	\end{equation*}
	where the  constant $C$ depends on the NTA character of $\Omega_1,\Omega_2$, and on $c_1$.
\end{lemma}

\begin{comment}
***only if we update the statement of the lemma to be for NTA domains***
    \begin{remark}\label{r:sharp hypoth}
    We emphasize that in particular, Lemma \ref{lem: bounding by logs CAD} holds in the case where $\Omega_1,\Omega_2\subset \mathbb{R}^{n+1}$ are disjoint chord-arc domains.
    \end{remark}
\end{comment}

\begin{proof}
	For $i=1,2$ denote by $g_i$ the Green function for $\Omega_i$ and define the functions $u_i(y)=g_i(y,p_i)$, where we take $u_i$ to be zero outside of $\Omega_i$. Since the domains $\Omega_i$ are NTA, they are Wiener regular. This
     guarantees that the functions $u_i$ are continuous away from $p_i$, for $i=1,2$. For all $x\in \mathbb{R}^{n+1}\setminus \left(\Omega_1\cup \Omega_2\right)$ and $t\in (2\delta_x,\frac{1}{4}\min_i(\dist(x,p_i)))$ the  Alt-Caffarelli-Friedman monotonicity formula (Theorem \ref{teoACF-elliptic}, \eqref{eqprec1}) yields
	\begin{align*}
		\frac{\partial_rJ(x,t)}{J(x,t)}\geq \frac{2}{t}(\alpha_1(x,t)+\alpha_2(x,t)-2),
	\end{align*}
	where 
	\begin{equation*}
		J(x,t)=\left(\frac{1}{t^2}\int_{B(x,t)}\frac{|\nabla u_1(y)|^2}{|y-x|^{n-1}}dy\right)\left(\frac{1}{t^2}\int_{B(x,t)}\frac{|\nabla u_2(y)|^2}{|y-x|^{n-1}}dy\right).
	\end{equation*}
 
	Fix $x\in \mathbb{R}^{n+1}\setminus(\Omega_1\cup \Omega_2)$ and $r\in (2\delta_x,\frac{1}{4}\min_i(\dist(x,p_i)))$. Then for $\rho\in (2\delta_x,r)$, 
	\begin{equation}\label{e: ACF bound on alphas}
		\int_{\rho}^{r}\frac{\alpha_1(x,t)+\alpha_2(x,t)-2}{t}\,dt\leq \int_{\rho}^{r}\frac{\partial_rJ(x,t)}{J(x,t)}\,dt
		=\log\left(\frac{J(x,r)}{J(x,\rho)}\right).
	\end{equation}

	Since $J(x,t)$ is increasing, we have $J(x,\rho)\leq J(x,r)$, and thus $1\leq \frac{J(x,r)}{J(x,\rho)}$. In particular, the right hand side of \eqref{e: ACF bound on alphas} is always nonnegative. 
    
    Let us first bound the numerator, $J(x,r)$. From \eqref{eqaux*} and the Caccioppoli inequality we obtain
	\begin{align}\label{e: initial bound on J(x,r)}
		\begin{split}
			J(x,r)
			&\lesssim_n \left(\frac{1}{r^2}\fint_{B(x,2r)}(u_1)^2dy \right)\left(\frac{1}{r^2}\fint_{B(x,2r)}(u_2)^2dy \right).
		\end{split} 
	\end{align}
	For a more detailed computation see \cite[Section 3]{KPT09}. From Lemmas \ref{lem:bound on harmonic measure corkscrew point} and \ref{lem:density of harmonic measure unif bounded} and the doubling property of $\omega_i$ (see \cite[Lemma 4.9]{JK82}), we have
	\begin{equation*}
		\frac{1}{r^2}\fint_{B(x,2r)}(u_i)^2 dy
		\lesssim_n \left(\frac{\omega_i(B(x,r))}{r^{n}}\right)^2,
	\end{equation*}
	and thus continuing the estimate in \eqref{e: initial bound on J(x,r)} gives 
	\begin{equation}\label{e: final bound on J(x,r)}
		J(x,r)\lesssim \left(\frac{\omega_1(B(x,r))}{r^n} \right)^2\left(\frac{\omega_2(B(x,r))}{r^n} \right)^2 = \theta_{1}(x,r)^2\,\theta_{2}(x,r)^2.
	\end{equation}
	We now lower bound for $J(x,\rho)$ for any $\rho\in (2\delta_x,r)$. Let $\varphi_{x,\rho}$ be a $C^{\infty}(\mathbb{R}^{n+1})$ bump function satisfying
	\begin{equation}\label{e: auxiliary phi}
		\chi_{B(x,\rho/2)}\leq \varphi_{x,\rho}\leq \chi_{B(x,\rho)}\quad \text{with}\quad \|\nabla \varphi_{x,\rho}\|_{\infty}\lesssim \rho^{-1}.
	\end{equation}
	By integration by parts, properties of the Green's function, and H\"{o}lder's inequality,
	\begin{equation*}
		\omega_i(B(x,\rho/2))\leq \int \varphi_{x,\rho}d\omega_i=-\int \nabla\varphi_{x,\rho}\nabla u_i dy\leq \|\nabla \varphi_{x,\rho}\|_{L^2(B(x,\rho))}\|\nabla u_i\|_{L^2(B(x,\rho))}.
	\end{equation*}
	We now continue this estimate and use \eqref{e: auxiliary phi} to obtain
	\begin{align*}
		\omega_i(B(x,\rho/2))& \lesssim_n \rho^{\frac{n-1}{2}}\left(\int_{B(x,\rho)}\frac{\rho^{n-1}}{|x-y|^{n-1}}|\nabla u_i|^2dy\right)^{1/2}  \\
        &
		\approx_n \rho^{n}\left(\frac{1}{\rho^2}\int_{B(x,\rho)}\frac{|\nabla u_i|^2}{|x-y|^{n-1}}dy\right)^{1/2}.
	\end{align*}
	Together with the doubling property of $\omega_i$ gives the following lower bound on $J(x,\rho)$,
	\begin{equation}\label{e: final bound on J(x,rho)}
		\left(J(x,\rho)\right)^{1/2}\gtrsim \frac{\omega_1(B(x,\rho))}{\rho^n}\cdot \frac{\omega_2(B(x,\rho))}{\rho^n}\gtrsim_n \theta_1(x,\rho)\cdot \theta_2(x,\rho).
	\end{equation}
	
	Combining \eqref{e: final bound on J(x,r)} and \eqref{e: final bound on J(x,rho)} yields
	\begin{align}\label{e: bound by sum of logs}
		\int_{\rho}^{r}\frac{\alpha_1(x,t)+\alpha_2(x,t)-2}{t}\,dt
		&\leq\log\left(\frac{\theta_{1}(x,r)}{\theta_1(x,\rho)}\right) +\log\left(\frac{\theta_{2}(x,r)}{\theta_2(x,\rho)}\right)+C.
	\end{align}
\end{proof}

\begin{lemma}\label{lem2}
Let $\Omega$ be a bounded chord-arc domain and let $p\in\Omega$ such that $\dist(p,\pom)\geq c_1 \diam(\pom)$. Denote by $\omega$ 
the harmonic measure for $\Omega$ with respect to the pole $p$. Then, for any ball $B$ centered in $\pom$ with radius $0<r(B)\leq \diam(\pom)$ and any Borel function $\rho:\partial\Omega\cap B\to (0,r(B))$,
we have
$$\int_{\pom \cap B} \log^+\!\left(\frac{\theta_{\omega}(x,r(B))}{\theta_{\omega}(x,\rho(x))}\right)\,d\sigma(x) \lesssim r(B)^n,$$
where the implicit constant depends on $c_1$ and the chord-arc character of $\Omega$.
\end{lemma}

%Remark that, for a chord-arc domain $\Omega$, the limit in the above definition of $\theta_{\omega}(x,0)$ exists $\sigma$-a.e.\ in $\partial\Omega$ because  $\partial\Omega$ is $n$-rectifiable and $\omega$ is absolutely continuous with respect to $\sigma$.
\vv

\begin{proof}[Proof of Lemma \ref{lem2}] Fix $B$ centered on $\partial \Omega$. Let $d=r(B)$ and $\theta(x,r):=\theta_{\omega}(x,r)$.
	For $x\in \partial \Omega \cap B$  we have
	\begin{align*}\label{e: first bound on log}
		\begin{split}
            \frac{\theta(x,d)}{\theta(x,\rho(x))}\lesssim \frac{\theta(3B)}{\theta(x,\rho(x))}
			=\frac{1}{\omega(B(x,\rho(x)))}\int_{B(x,\rho(x))}\theta(3B)\,\frac{d\sigma}{d\omega}\, d\omega
		\end{split}
	\end{align*}
	where $\frac{d\sigma}{d\omega}$ is the Radon-Nikodym derivative of $\sigma$ with respect to $\omega$. Let 
	\begin{equation}\label{e: definition of f^+}
		f:=\theta(3B)\,\frac{d\sigma}{d\omega}\,\chi_{3B}, 
	\end{equation}
	so that
	\begin{equation}\label{e: bound by HL max op}
		\frac{\theta(x,d)}{\theta(x,\rho(x))}\lesssim \frac{1}{\omega(B(x,\rho(x)))}\int_{B(x,\rho(x))}f\, d\omega \leq M_{\omega}f(x), 
	\end{equation}
	where $M_{\omega}f(x)=\sup_{r>0}\frac{1}{\omega(B(x,r))}\int_{B(x,r)}fd\omega$.

	Observe that $\omega$ is doubling and since $\sigma\in A_{\infty}(\omega)$, there exists $s^{\prime}>1$ such that the following reverse H\"{o}lder inequality holds: 
	\begin{equation}\label{e: RHI with sprime}
		\frac{1}{\omega(3B)}\int_{3B}f^{s^{\prime}}d\omega
		\leq \left(\frac{c}{\omega(3B)}\int_{3B}fd\omega\right)^{s^{\prime}}. 
	\end{equation}
    Now for $q=s^{\prime}-1>0$, by \eqref{e: bound by HL max op} and the fact that $\log^+\!(t)\lesssim|t|^q$,  
        \begin{align*}
            \int_B \log^+\!\left(\frac{\theta(x,d)}{\theta(x,\rho)}\right)d\sigma 
            &\leq \int_{B}\log^+\! \left(C\,M_{\omega}f\right)d\sigma\\
		&\lesssim \int_{B}\left(M_{\omega}f\right)^q\frac{d\sigma}{d\omega}\,d\omega\\
        & 
		\leq \left(\int_{B}(M_{\omega}f)^{qs}d\omega\right)^{1/s}\left(\int_{B}\left(\frac{d\sigma}{d\omega}\right)^{s^{\prime}}d\omega\right)^{1/s^{\prime}},
        \end{align*}
        where $s'$ is the conjugate exponent of $s$. Then, since
         $qs=s^{\prime}$, 
	\begin{align*}
		\int_B \log^+\!\left(\frac{\theta(x,d)}{\theta(x,\rho)}\right)d\sigma &\lesssim \left(\int_{3B}f^{qs}d\omega\right)^{1/s}\left(\int_{B}\left(\frac{d\sigma}{d\omega}\right)^{s^{\prime}}d\omega\right)^{1/s^{\prime}}\\
		&\approx\left(\int_{3B}f^{s^{\prime}}d\omega\right)^{1/s}\left(\int_{B}\left(\frac{d^n}{\omega(3B)}f\right)^{s^{\prime}}d\omega\right)^{1/s^{\prime}}\\
		&=\frac{d^n}{\omega(3B)}\left(\int_{3B}f^{s^{\prime}}d\omega\right)^{1/s}\left(\int_{B}f^{s^{\prime}}d\omega\right)^{1/s^{\prime}}\\
		&\lesssim \frac{d^n}{\omega(3B)} \int_{3B} f^{s^{\prime}}d\omega.
	\end{align*}
	 The result then follows from \eqref{e: definition of f^+} and \eqref{e: RHI with sprime}.
\end{proof}

We remark that from Lemmas \ref{lem: bounding by logs CAD} and \ref{lem2} it follows that in the particular case of chord-arc domains, Theorem \ref{t:main2} (3) holds. In the next section we will show how a corona decomposition of two-sided corkscrew domains in terms of CADS allow us to prove the same for this latter, more general, class of domains. 

%begin with two-sided corkscrew domain with $n$-ADR boundary (a more general class of domains) and obtain the same result using the estimates from 

\section{Corona decomposition into Lipschitz subdomains}\label{s:corona}

In this section we finish the proof of Theorem \ref{t:main2}. Recall: the only task left was to show that $(1) \implies (3)$. The plan, then, is to construct a multiscale decomposition of $\Omega_1$ and $\Omega_2$ into Lipschitz subdomains. These domains are in particular CAD, and therefore we will be in the position to apply the estimates proven in the previous section.

\subsection{The corona decomposition using Lipschitz subdomains}\label{secorona}

 In this section we assume that $\Omega^+\equiv\Omega_1\subset\R^{n+1}$ is a two-sided corkscrew open set with uniformly $n$-rectifiable
boundary. We denote $\Omega^-=\Omega_2=\R^{n+1}\setminus\overline{\Omega^+}$ and we let $\sigma=\HH^n|_{\partial\Omega^+}$.

\vv

\subsubsection{The approximating Lipschitz graph}\label{subs:approxLipgraph}

In this subsection we describe how to associate an approximating Lipschitz graph to a cube $Q\in \DD_\sigma$, assuming $b\beta_{\sigma}(k_1Q)$ to be small enough for some big constant $k_1>20$ {(where we denoted $b\beta_\sigma\equiv b\beta_{\supp\sigma}$).}
We will follow 
the arguments in \cite{MT21} quite closely, which in turn are based on \cite[Chapters 7, 8, 12, 13, 14]{DS91}.
The first step consists in defining suitable stopping cubes.

Given $x\in\R^{n+1}$, we write $x= (x',x_{n+1})$.
For a given cube $Q\in\DD_\sigma$, we denote by $L_Q$ a best approximating hyperplane for
 $b\beta_\sigma(k_1Q)$.
We also assume, without loss of generality, that 
$$
L_Q \,\,\textup{is the horizontal  hyperplane}\,\, \{x_{n+1}=0\}.
$$ 
We denote by ${C}(Q)$ the cylinder
$${C}(Q) = \big\{x\in \R^{n+1}: |x'-(x_Q)'|\leq 10\,\ell(Q),\,
{|x_{n+1}-(x_Q)_{n+1}| }\leq 10\,\ell(Q)\big\}.$$
Observe that $C(Q)\subset 20B_Q$.

We fix $0<\ve\ll\delta\ll1$ to be chosen later (depending on the corkscrew condition and the uniform rectifiability constants), $k_1>20$, and we denote by $\cB$ or $\cB(\ve)$ the family of cubes $Q\in
\DD_\sigma$ such that 
$b\beta_\sigma(k_1Q) > \ve$. For a given cube 
 $Q\in\DD_\sigma$ such that $b\beta_\sigma(k_1Q)\leq \ve$, we let $\sss(Q)$ be the family of maximal cubes $P\in\DD_\sigma$ which are contained in $k_1Q$ and such that at least one of the following holds:
\begin{itemize}
\item[(a)] {$P\cap C(Q) = \varnothing$.}
\item[(b)] $P\in\cB(\ve)$, i.e., $b\beta_\sigma(k_1P) > \ve$.
\item[(c)] $\angle(L_P,L_Q)> \delta$, where $L_P$, $L_Q$ are best approximating hyperplanes for $\beta_{\sigma,\infty}(k_1P)$ and $\beta_{\sigma,\infty}(k_1Q)$, respectively, and {$\angle(L_P,L_Q)$ denotes the angle between $L_P$ and $L_Q$.}
\end{itemize}
We denote by $\tree(Q)$ the family of cubes in $\DD_\sigma$ which are contained in $k_1Q$ and which are not strictly contained
in any cube from $\sss(Q)$. We also consider the function
$$d_Q(x) = \inf_{P\in\tree(Q)} \big(\dist(x,P) + \diam(P)\big).$$
Notice that $d_Q$ is $1$-Lipschitz.
Assuming $k_1$ big enough (but independent of $\ve$ and $\delta$) and arguing as in the proof of  \cite[Proposition 8.2]{DS91}, the following holds:

\begin{lemma}\label{lemgraf}
Denote by $\Pi_Q$ the orthogonal projection on $L_Q$.
There is a Lipschitz function $A:L_Q \to L_Q^\bot$ with slope at most $C\delta$ such that
$$\dist(x,(\Pi_Q(x),A(\Pi_Q(x)))) \leq C_1\ve\,d_Q(x)\quad \mbox{ for all $x\in 20B_Q$.}$$
\end{lemma}

In this lemma, and in the whole subsection, we assume that {$Q$ is} as 
above, so that, in particular, $b\beta_\sigma(k_1Q)\leq \ve$.

We denote
$$D_Q(x)= \inf_{y\in\Pi_Q^{-1}(x)}d_Q(y).$$
It is immediate to check that $D_Q$ is also a $1$-Lipschitz function. Further, as shown in \cite[Lemma
8.21]{DS91}, there is some fixed constant $C_2$ such that
\begin{equation}\label{eqDR}
C_2^{-1}d_Q(x) \leq D_Q(x) \leq d_Q(x)\quad \mbox{ for all $x\in 20B_Q\cap \pom^+$.}
\end{equation}

We denote by $Z(Q)$ the set of points $x\in Q$ such that {$d_Q(x)=0$.}
The following lemma is an immediate consequence of {the} results obtained in  \cite[Chapters 7, 12-14]{DS91}. See also Lemma 3.2 from \cite{MT21}. 

\begin{lemma}\label{lempack1}
There are some constants $C_3(\ve,\delta)>0$ and $k\geq 2$ such that
\begin{equation}\label{eqlempack1}
\sigma(Q) \approx \sigma(C(Q))\leq 2\,\sigma(Z(Q)) + 2\!\sum_{P\in\sss(Q)\cap\cB(\ve)} \!\sigma(P) + C_3\! \sum_{P\in\tree(Q)} \!\beta_{\sigma,1}(kP)^2\,\sigma(P).
\end{equation}
\end{lemma}

\subsubsection{The Lipschitz  subdomains $\Omega_Q^\pm$}\label{subs:star-Lip}

 Abusing notation,  we write below 
 $$
 D_Q(x')=D_Q(x),\quad\textup{for}\,\,x=(x',x_{n+1}).
 $$

\begin{lemma}\label{lem333}
Let $$U_Q = \{x\in C(Q): x_{n+1}> A(x') + C_1C_2\ve D_Q(x')\},$$
 $$V_Q = \{x\in C(Q): x_{n+1}< A(x') - C_1C_2\ve D_Q(x')\}.$$ Then one of the sets $U_Q$, $V_Q$ is contained in $\Omega^+$ and the other in $
\Omega^-$.
\end{lemma}

\begin{proof}
Denote 
$$W_Q = \{x\in C(Q): A(x') - C_1C_2\ve D_Q(x') \leq x_{n+1}\leq A(x') + C_1C_2\ve D_Q(x')\}.$$
We claim that ${\partial\Omega^+}\cap C(Q) \subset W_Q$. Indeed, we have $\partial\Omega^+\cap C(Q)\subset \partial\Omega^+\cap \,20B_Q$.
%$\partial\Omega^+\cap C(Q)\subset \partial\Omega^+\cap B(Q) \subset Q$, by the definition of $B(Q)$.
Then, by Lemma \ref{lemgraf} and \rf{eqDR}, for all $x\in \partial\Omega^+\cap C(Q)$ we have
$$|x- (x',A(x'))|\leq C_1\ve\,d_Q(x) \leq C_1C_2\ve\,D_Q(x),$$
which is equivalent to saying that $x\in W_Q$.

Next we claim that if $U_Q\cap \Omega^+\neq \varnothing$, then $U_Q\subset \Omega^+$.
This follows from connectivity, taking into account that if $x\in U_Q\cap \Omega^+$ and $r=\dist(x,\partial U_Q)$, then $B(x,r)\subset \Omega^+$. Otherwise, there exists some point $x'\in B(x,r)\setminus
\overline \Omega^+$, and thus there exists some $x''\in{\partial\Omega^+}$ which belongs to the segment $\overline{x,x'}$. This would contradict the fact that ${\partial\Omega^+}\subset W_Q$.
The same argument works replacing $U_Q$ and/or $\Omega^+$ by $V_Q$ and/or $\R^{n+1}\setminus\overline{\Omega^+}$, and thus we deduce that any of the sets $U_Q$, $V_Q$ is contained either in $\Omega^+$ or in $\Omega^-=\R^{n+1}\setminus\overline{\Omega^+}$.

Finally suppose that one of the sets $U_Q,V_Q$, say $U_Q$, is contained in $\Omega^+$. From the two-sided
corkscrew condition we infer that there exists some exterior corkscrew point $y\in B(x_Q,r(B(Q))/2)\cap \Omega^-$ with $\dist(y,{\partial\Omega^+})\gtrsim r(B(Q))$. Recall that $r(B(Q))\approx\ell(Q)$, and thus $B(x_Q,r(B(Q))/2)\subset C(Q)$. So, if $\ve$ is small enough we deduce that $y\in (U_Q \cup V_Q)\cap \Omega^-$. Since $y$ cannot belong to $U_Q$, it belongs to $V_Q$, and thus $V_Q$ intersects $\Omega^-$. Then by the discussion in the previous paragraph, $V_Q\subset\Omega^-$.
\end{proof}
\vv

Suppose that $U_Q\subset \Omega^+$. For a given $\delta\in (0,1/100)$,
we denote by $\Gamma_Q^+$ the Lipschitz graph of the function 
$C(Q)\cap L_Q\ni x'\mapsto A(x') + \delta\,D_Q(x')$. Notice that this is a Lipschitz function with slope at most 
$C\delta< 1$ (assuming $\delta$ small enough). So $\Gamma_Q^+$ intersects neither the top nor the bottom faces of $C(Q)$, assuming $\ve$ small enough too.
Then we define
\begin{equation}\label{eqomega+r}
\Omega_Q^+ =\big\{x=(x',x_{n+1}) \in {\rm Int}(C(Q)): x_{n+1}> A(x') + \delta \,D_Q(x')\big\}.
\end{equation}
Observe that $\Omega_Q^+$ is a {starlike Lipschitz domain (with uniform Lipschitz character) }and that $\Omega_Q^+\subset U_Q$, assuming that $C_1C_2\ve\ll\delta$.

We define $\Gamma_Q^-$ and $\Omega_Q^-$ analogously, replacing the above function $A(x') + \delta\,D_Q(x')$ by $A(x') - \delta\,D_Q(x')$. 
From Lemma \ref{lem333} and the assumption that $C_1C_2\ve\ll\delta$,
it is immediate to check that
\begin{equation}\label{eqsep99}
 \dist(x,\partial\Omega^+)\geq \frac\delta2\,D_Q(x)\quad\mbox{ for all $x\in\Omega_Q^+\cup\Omega_Q^-$.}
\end{equation}
Without loss of generality, we will assume that $\Omega_Q^+\subset\Omega^+$ and $\Omega_Q^-\subset\Omega^-$.

\vv
\subsubsection{The corona decomposition of ${\partial\Omega^+}$}\label{subs:cordecOmega}

For any $Q\in\DD_\sigma$ we define $\nex(Q)$ as follows:
\begin{itemize}
\item If $Q\not \in \cB(\ve)$ (i.e., $b\beta_{\sigma}(k_1Q)\leq \ve$), we let $\nex(Q)$ be the family of cubes which belong to $\Ch(P)$ for some $P\in\sss(Q)\cap\DD_\sigma(Q)$.
\item If $Q\in \cB(\ve)$ (i.e., $b\beta_{\sigma}(k_1Q)> \ve$), we let $\nex(Q) = \Ch(Q)$.
\end{itemize}
Notice that the cubes from $\nex(Q)$ are contained in $Q$.

Let $R_0\in\DD_\sigma$. We define a family $\ttt(R_0)\subset \DD_\sigma(R_0)$ inductively as follows. First we set $\ttt_0(R_0)=\{R_0\}$.
Assuming $\ttt_k(R_0)$ to be defined, we set
$$\ttt_{k+1}(R_0) = \bigcup_{Q\in\ttt_k(R_0)} \nex(Q).$$
We set $$\ttt(R_0)=\bigcup_{k\geq0} \ttt_k(R_0).$$

\vv

\begin{lemma}\label{lempack2}
The family $\ttt(R_0)$ satisfies the packing condition
$$\sum_{Q\in\ttt(R_0):Q\subset R_0} \sigma(Q)\lesssim_{\ve,\delta} \sigma(R_0).$$
\end{lemma}

The proof of this lemma is standard, using \rf{lempack1} and the uniform rectifiability of ${\partial\Omega^+}$. See for example Lemma 3.8 from \cite{MT21} for a related argument.

\vv
    
\subsection{The main estimate}\label{s:mainestimate}
Here, we will use the multiscale decomposition constructed above to transfer the good estimates that hold, by Lemma \ref{lem2}, for the approximating Lipschitz (and thus CAD) domains onto $\Omega^\pm$ themselves.

Precisely, we aim to prove the following:

\begin{proposition}
Let $\Omega^+$ be a bounded two-sided corkscrew domain and suppose $\partial \Omega^+$ is $n$-Ahlfors regular. %and let $p\in\Omega^+$ be such that $\dist(p,\pom^+)\geq c_1 \diam(\pom^+)$.
Let $\xi\in\pom^+$ and $0<r\leq \diam(\pom^+)$. Then
$$\int_{B(\xi,r)}\int_0^r a(x,t)\,\frac{dt}t \,d\sigma(x)\lesssim r^n.$$
\end{proposition} 

\begin{proof}
Recall first that, by \cite{DJ90} and \cite{Semmes90}, the fact that $\Omega^+$ is two-sided corkscrew open set with $n$-Ahlfors regular boundary
implies that $\pom^+$ is uniformly $n$-rectifiable.

Denote by $I_{\xi,r}$ the family of  cubes from $\DD_\sigma$ which intersect $B(\xi,r)$ having side length at most $8r$ and such that moreover
they are maximal. Observe that this implies that their side length is at least $4r$.
Since the cubes from $I_{\xi,r}$ have side length comparable to $r$, it follows easily that $\# I_{\xi,r}\lesssim1$.

For each $R\in I_{\xi,r}$ we consider the family $\ttt(R)$ constructed in the preceding section. Then, for any $x\in R\in I_{\xi,r}$, we have
$$\int_0^r a(x,t)\,\frac{dt}t \leq \sum_{Q\in\ttt(R):x\in Q} \int_{\ell_Q(x)}^{\ell(Q)} a(x,t)\,\frac{dt}t, 
$$
where $\ell_Q(x)$ is the side length of the cube from $\nex(Q)$ that contains $x$, and we set $\ell_Q(x)=0$ if that cube does not exist.
Then we get
\begin{align}\label{eqsumfac8}
\int_R\int_0^r \!a(x,t)\,\frac{dt}t\,d\sigma(x) & \leq  \int_R\sum_{Q\in\ttt(R):x\in Q} \int_{\ell_Q(x)}^{\ell(Q)} a(x,t)\,\frac{dt}t\,d\sigma(x)\nonumber\\
& 
=\!
\sum_{Q\in\ttt(R)}
\int_Q \int_{\ell_Q(x)}^{\ell(Q)} a(x,t)\,\frac{dt}t\,d\sigma(x).
\end{align}

If $Q\in\ttt(R)\cap\cB(\ve)$, then $\nex(Q)=\Ch(Q)$ and thus $\ell_Q(x)=\ell(Q)/2$ for all $x\in Q$. Therefore, we can estimate
\begin{equation}\label{eqsumfac9}
\int_Q \int_{\ell_Q(x)}^{\ell(Q)} a(x,t)\,\frac{dt}t\,d\sigma(x) \leq \int_Q \int_{\ell(Q)/2}^{\ell(Q)} 1\,\frac{dt}t\,d\sigma(x)\lesssim
\sigma(Q).
\end{equation}

In the case $Q\in\ttt(R)\setminus\cB(\ve)$, we consider the associated Lipschitz domains $\Omega_Q^+$ and $\Omega_Q^-$ constructed in \rf{eqomega+r}. 
We denote by $\omega^\pm_Q$ the respective harmonic measures for $\Omega_Q^\pm$ with respect to poles $p_Q^\pm\in \Omega_Q^\pm$ such that
$\dist(p_Q^\pm,\partial\Omega_Q^\pm)\geq c_2\ell(Q)\approx\diam(\Omega_Q^\pm).$ Since $a(x,t)\leq1$, for $c_3=c_2/2$ and for any $x\in Q$ we have
$$\int_{c_3\ell(Q)}^{\ell(Q)} a(x,t)\,\frac{dt}t\lesssim 1.$$
So we can write
\begin{align}\label{eqsplit41}
\int_Q \int_{\ell_Q(x)}^{\ell(Q)} a(x,t)\,\frac{dt}t\,d\sigma(x) 
& \leq \int_Q \int_{\ell_Q(x)}^{c_3\ell(Q)} a(x,t)\,\frac{dt}t\,d\sigma(x) + C\sigma(Q)\\
& = \sum_{P\in\sss(Q)\cap \DD_\sigma(Q)} \int_P\int_{\ell(P)/2}^{c_3\ell(Q)} a(x,t)\,\frac{dt}t\,d\sigma(x) \nonumber\\
& \quad+ \int_{Z(Q)}\int_0^{c_3\ell(Q)} a(x,t)\,\frac{dt}t\,d\sigma(x) + C\sigma(Q).\nonumber
\end{align}

Notice that if $x\in Q\subset\pom^+$, then $x\in\R^{n+1}\setminus(\Omega^+_Q\cup\Omega^-_Q)$. 
Since $\Omega_Q^\pm\subset \Omega^\pm$, for any $t>0$ we have 
$$\alpha_{\Omega^\pm}(x,t) \leq \alpha_{\Omega_Q^\pm}(x,t),$$
understanding that $\alpha_{\Omega_Q^\pm}(x,t)=\infty$ if $\partial B(x,t)\cap \Omega_Q^\pm=\varnothing$.
So denoting $$a_Q(x,t) = \min\big(\alpha_{\Omega_Q^+}(x,t) + \alpha_{\Omega_Q^-}(x,t)-2,\,1\big),$$
it follows that 
$$a(x,t) \leq a_Q(x,t).$$
Together with Lemma \ref{lem: bounding by logs CAD}, this gives
\begin{equation}\label{eqguai16}
  \int_{\rho(x)}^{r}a(x,t) \frac{dt}t\leq      \int_{\rho}^{r}a_Q(x,t) \frac{dt}t\lesssim \log^+\!\left(\frac{\theta_{\omega^+_Q}(x,r)}{\theta_{\omega^+_Q}(x,\rho(x))}\right) + \log^+\!\left(\frac{\theta_{\omega^-_Q}(x,r)}{\theta_{\omega^-_Q}(x,\rho(x))}\right) + 1,     
   \end{equation}
for  $\rho(x)$, $r$ such that $2\delta_x \leq \rho(x)\leq r \leq \min_i(\dist(x,p_i))$, with $\delta_x= \max_{i=\pm}(\dist(x,\pom_Q^i))$.

Notice that $Z(Q) \subset \pom\cap \partial\Omega_Q^+\cap \partial\Omega_Q^-$. Notice that the densities
$$\theta_{\omega^\pm_Q}(x,0) = \lim_{r\to 0} \frac{\omega^\pm_Q(x,r)}{r^n}$$
exist $\sigma$-a.e.\ in $Z(Q)$ because $\omega^\pm_Q$ is mutually absolutely continuous with $\HH^n\rest \partial\Omega_Q^\pm$ and $\partial\Omega_Q^\pm$ is $n$-rectifiable. Thus, we deduce that
\begin{align}\label{eqzq+-}
\int_{Z(Q)}\int_0^{c_3\ell(Q)} \!a(x,t)\,\frac{dt}t\,d\sigma(x) & \lesssim  
\int_{\partial\Omega_Q^+} \log^+\!\left(\frac{\theta_{\omega^+_Q}(x,c_3\ell(Q))}{\theta_{\omega^+_Q}(x,0)}\right) \,d\HH^n(x)\\
&\quad \!+
\int_{\partial\Omega_Q^-} \!\log^+\!\left(\frac{\theta_{\omega^-_Q}(x,c_3\ell(Q))}{\theta_{\omega^-_Q}(x,0)}\right) d\HH^n(x) + C\,\sigma(Q).\nonumber
\end{align}
%Above we denoted $\theta_{\omega^\pm_Q}(x,0)=\lim_{r\to0}\theta_{\omega^\pm_Q}(x,r)$. The limit exists $\HH^n$-a.e.\ because $\omega^\pm_Q$ is absolutely continuous with respect to $\HH^n$ in $\partial\Omega_Q^\pm$.

To deal with the first term on the right hand side of \rf{eqsplit41}, we will associate a subset $\Delta_P^\pm\subset \partial\Omega_Q^\pm$ to each $P\in\sss(Q) \cap \DD_\sigma(Q)$.  Observe first that if $P,P'\in\sss(Q)$, then 
\begin{equation}\label{eqbigdistP}
|x_P- x_{P'}|\geq c_4(\ell(P) + \ell(P')),
\end{equation}
for some constant $c_4>0$ depending on the properties of the dyadic lattice $\DD_\sigma$. Then we define
$$\Delta_P^\pm = B(x_P,c_4\ell(P)/2) \cap \partial\Omega_Q^\pm\quad \mbox{ for each $P\in\sss(Q) \cap \DD_\sigma(Q)$.}$$
From \rf{eqbigdistP}, it follows easily that $\Delta_P^+\cap \Delta_{P'}^+=\varnothing$ if $P,P'$ are different cubes from $\sss(Q)\cap \DD_\sigma(Q)$, and the same happens for $\Delta_P^-, \Delta_{P'}^-$.
Notice now that by Lemma \ref{lemgraf}, \rf{eqDR}, and the definitions  $\Omega^+_Q$, $\Omega^-_Q$, and $d_Q$, for any $y\in P$ we have
$$\dist(y,\partial\Omega_Q^\pm) \lesssim \delta\,(d_Q(y) + D_Q(y))\approx \delta\,d_Q(y) \lesssim \delta\,\ell(P).$$
In particular, the center $x_P$ of $P$ satisfies $\dist(x_P,\partial\Omega_Q^\pm) \leq C\delta\,\ell(P)$.
Hence, if $\delta$ is taken small enough, then $B(x_P,c_4\ell(P)/2)$ intersects a big portion of $\partial\Omega_Q^\pm$ and 
it follows that
$$\HH^n(\Delta_P^\pm) \gtrsim \ell(P)^n,$$
by the Ahlfors regularity of $\partial\Omega_Q^\pm$.

For each $P\in\sss(Q)\cap \DD_\sigma(Q)$, by \rf{eqguai16} and  
thanks to the properties of $\Delta_P^\pm$, we have
\begin{align*}
\int_P\int_{\ell(P)/2}^{c_3\ell(Q)} a(x,t)\,\frac{dt}t\,d\sigma(x) &
\lesssim \int_P
   \log^+\!\left(\frac{\theta_{\omega^+_Q}(x,\ell(Q))}{\theta_{\omega^+_Q}(x,\ell(P))}\right) \,d\sigma(x)\\
   & \,\,\,\,\,\,\,\,+ \int_P \log^+\!\left(\frac{\theta_{\omega^-_Q}(x,\ell(Q))}{\theta_{\omega^-_Q}(x,\ell(P))}\right)\,d\sigma(x) + \sigma(P).  
 \end{align*}
 
Using now that $\omega_Q^+$ is doubling and that $\sigma(P)\approx \HH^n(\Delta_P^+)$, we derive
\begin{align*}
\int_P
   \log^+\!\left(\frac{\theta_{\omega^+_Q}(x,\ell(Q))}{\theta_{\omega^+_Q}(x,\ell(P))}\right) \,d\sigma(x) &
\lesssim  \inf_{x\in \Delta_P^+}
   \log^+\!\left(\frac{\theta_{\omega^+_Q}(x,\ell(Q))}{\theta_{\omega^+_Q}(x,\ell(P))}\right) \,\sigma(P) + \sigma(P)\\
&   \lesssim \int_{\Delta_P^+}
   \log^+\!\left(\frac{\theta_{\omega^+_Q}(x,\ell(Q))}{\theta_{\omega^+_Q}(x,\ell(P))}\right) \,d\HH^n(x) + \sigma(P).
\end{align*}
The same estimate holds replacing $\omega^+$ and $\Delta_P^+$ by $\omega^-$ and $\Delta_P^-$. Then we deduce
\begin{align*} 
&\sum_{P\in\sss(Q)\cap \DD_\sigma(Q)} \int_P\int_{\ell(P)/2}^{c_3\ell(Q)} a(x,t)\,\frac{dt}t\,d\sigma(x) \\
&
\lesssim \sum_{P\in\sss(Q)\cap \DD_\sigma(Q)} \int_{\Delta_P^+}
   \log^+\!\left(\frac{\theta_{\omega^+_Q}(x,\ell(Q))}{\theta_{\omega^+_Q}(x,\ell(P))}\right) \,d\HH^n(x) \\
   & \hspace{3cm}+ \sum_{P\in\sss(Q)\cap \DD_\sigma(Q)}
   \int_{\Delta_P^-}
   \log^+\!\left(\frac{\theta_{\omega^-_Q}(x,\ell(Q))}{\theta_{\omega^-_Q}(x,\ell(P))}\right) \,d\HH^n(x)+ \sigma(Q)\\
   & \leq 
    \int_{\partial\Omega_Q^+}
   \log^+\!\left(\frac{\theta_{\omega^+_Q}(x,\ell(Q))}{\theta_{\omega^+_Q}(x,2\ell_Q(x))}\right) \,d\HH^n(x)\\
   & \hspace{3cm}+ 
   \int_{\partial\Omega_Q^-}
   \log^+\!\left(\frac{\theta_{\omega^-_Q}(x,\ell(Q))}{\theta_{\omega^-_Q}(x,2\ell_Q(x))}\right) \,d\HH^n(x) + \sigma(Q).
\end{align*}   
   
From \rf{eqsplit41}, \rf{eqzq+-}, the preceding estimate, and Lemma \ref{lem2} applied to $\Omega_Q^\pm$, we get  
\begin{align*} 
\int_Q &\int_{\ell_Q(x)}^{\ell(Q)} a(x,t)\,\frac{dt}t\,d\sigma(x)\\
&\lesssim
\int_{\partial\Omega_Q^+}
   \log^+\!\left(\frac{\theta_{\omega^+_Q}(x,\ell(Q))}{\theta_{\omega^+_Q}(x,2\ell_Q(x))}\right) \,d\HH^n(x)\\
   & \hspace{3cm}+ 
   \int_{\partial\Omega_Q^-}
   \log^+\!\left(\frac{\theta_{\omega^-_Q}(x,\ell(Q))}{\theta_{\omega^-_Q}(x,2\ell_Q(x))}\right) \,d\HH^n(x) + \sigma(Q)
\lesssim \sigma(Q).
\end{align*}
By \rf{eqsumfac8}, \rf{eqsumfac9}, the preceding estimate, and the packing condition in Lemma \ref{lempack2}, we get
\begin{align*}
\int_R\int_0^r \!a(x,t)\,\frac{dt}t\,d\sigma(x)& \leq \sum_{Q\in\ttt(R)}
\int_Q \int_{\ell_Q(x)}^{\ell(Q)} a(x,t)\,\frac{dt}t\,d\sigma(x)\\
& \lesssim \sum_{Q\in\ttt(R)}\sigma(Q) \lesssim \sigma(R).
\end{align*}
Using now that $\# I_{\xi,r}\lesssim1$, it follows that
$$\int_{B(\xi,r)} \int_0^r a(x,t)\,\frac{dt}t\,d\sigma(x)\leq \sum_{R\in I_{\xi,r}}\int_R\int_0^r \!a(x,t)\,\frac{dt}t\,d\sigma(x) \lesssim
\sum_{R\in I_{\xi,r}}\sigma(R) \lesssim r^n.$$
\end{proof}

\bibliographystyle{alpha}
\bibdata{references}
\bibliography{references}
\end{document}